\documentclass[12pt,a4paper]{amsart}
\usepackage{verbatim,fancyvrb}
\usepackage{setspace}
\usepackage[left=2.5cm,right=2.5cm,top=3.0cm,bottom=3.0cm,includeheadfoot]{geometry}
\usepackage{latexsym}
\usepackage{amssymb}
\usepackage{amsthm}
\usepackage{amsmath}
\usepackage{tikz}
\colorlet{changes}{teal}
\usepackage{enumitem}
\usetikzlibrary{matrix}
\usetikzlibrary{positioning}
\usetikzlibrary{arrows}
\usepackage{graphicx}
\usepackage[font=small,labelfont=bf]{caption}
\usepackage[normalem]{ulem}

\usepackage[latin1]{inputenc}
\usepackage{hyperref}
\usepackage[normalem]{ulem}

\usepackage[english]{babel}

\newcommand{\GL}{\mathrm{GL}}

\newcommand{\SL}{\mathrm{SL}}

\newcommand{\mft}{\mathfrak{t}}

\newcommand{\rk}{\mathrm{rk}}

\newcommand{\RR}{\mathbb{R}}
\newcommand{\CC}{\mathbb{C}}
\newcommand{\ZZ}{\mathbb{Z}}
\newcommand{\QQ}{\mathbb{Q}}

\newcommand{\NN}{\mathbb{N}}

\newtheorem{thm}{Theorem}[section]
\newtheorem{prop}[thm]{Proposition}
\newtheorem{cor}[thm]{Corollary}
\newtheorem{lem}[thm]{Lemma}

\theoremstyle{plain}
\newtheorem{defn}[thm]{Definition}

\newtheorem{lemma}[thm]{Lemma}

\newtheorem{maintheorem}{Theorem}

\theoremstyle{definition}
\newtheorem{rem}[thm]{Remark}
\newtheorem{ex}[thm]{Example}

\title[Finiteness and boundedness of positive monotone Hamiltonian GKM$_3$ spaces]{Finiteness and boundedness of positive monotone Hamiltonian GKM$_3$ spaces}

\author[P. Konstantis]{Panagiotis Konstantis}
\address{Fachbereich Mathematik und Informatik, Philipps Universit\"at Marburg,
Hans-Meerwein-Str. 6, 35043 Marburg, Germany}
\email{pako@mathematik.uni-marburg.de}

\author[S. Sabatini]{Silvia Sabatini}
\address{Department Mathematik/Informatik, Universit\"at zu K\"oln,
  Weyertal 86-90, 50931 K\"oln, Germany.}
\email{sabatini@math.uni-koeln.de}

\author[L. Zoller]{Leopold Zoller}
\address{Facultat de Matem\`atiques i Inform\`atica, Gran Via de les Corts Catalanes 585, 08007 Barcelona, Spain}
\email{leopold.zoller@ub.edu}

\date{\today}

\begin{document}

\begin{abstract}

In this paper, we establish several finiteness and boundedness theorems for compact positive monotone symplectic manifolds endowed with special Hamiltonian actions, called GKM$_3$, which generalize smooth toric varieties. Specifically, we prove that, for fixed dimension and Euler characteristic, there are only finitely many complex cobordism classes of such spaces. 
Moreover, modulo lattice transformations, the moment map image can be embedded into a box of explicitly bounded size, and all Chern numbers satisfy 
quantitative bounds. In particular, this yields a bound on the volume of the underlying symplectic manifold, analogous to the one obtained by Koll\'{a}r-Miyaoka-Mori for Fano varieties.
As a byproduct of our results, we also obtain that there are finitely many combinatorially maximal, labeled, reflexive GKM$_3$ graphs of valency $n$ and with $\chi$ vertices.

 \end{abstract}

\maketitle

\tableofcontents

\section{Introduction}

A compact complex manifold $(M,J)$ is called \textit{Fano} if its anticanonical bundle is ample. This is equivalent to saying that there exists a K\"ahler form $\omega\in \Omega^{1,1}(M)$
such that $c_1=[\omega]$, where $c_1$ denotes the first Chern class of the tangent bundle of $(M,J)$. 

The latter definition motivates the following generalization to symplectic manifolds:
Let $(M,\omega)$ be a compact symplectic manifold and consider an almost complex structure $J$ compatible with $\omega$. 
Since the set of such structures is contractible, the Chern classes $c_j\in H^{2j}(M;\ZZ)$ of the tangent bundle $(TM,J)$ are indeed invariants of $(M,\omega)$.
\begin{defn}
A compact symplectic manifold $(M,\omega)$ is called \textbf{positive monotone} if
\begin{equation}\label{positive monotone}
c_1=[\omega]
\end{equation} 
\end{defn}
Positive monotone symplectic manifolds can therefore be viewed as a generalization of Fano varieties, and it is natural to ask to what extent these two classes of manifolds exhibit similar behavior and in what respects they differ. For instance, it is known that Fano varieties are simply connected (see \cite[Corollary 6.2.18]{MR1668579}) and that, by the celebrated 
Koll\'ar-Miyaoka-Mori theorem \cite[Theorem 0.2]{MR1189503}, there exists a quantitative upper bound for their volume $c_1^n[M]$ in each complex dimension $n\in \NN$. In particular 
Fano varieties of a fixed dimension $n\in \NN$ form a bounded family. This implies that, in each dimension, there are only finitely many diffeomorphism types, as well as finitely many complex cobordism classes. 

In contrast, in each fixed dimension $2n\geq 12$, Fine and Panov \cite{MR2679581} constructed compact positive monotone symplectic manifolds that are not simply connected; moreover, in each fixed dimension $2n\geq 12$, it is possible to produce a family of (non simply connected) compact positive
monotone symplectic manifolds whose volumes tend to infinity (see \cite[Section 5]{MR2679581} and Remark \ref{fine panov}).

In the presence of certain Hamiltonian symmetries, the two families of compact positive monotone and Fano manifolds seem to have a much more similar behavior.  
A conjecture of Fine and Panov from 2015 (\cite[Conjecture 1.4]{MR3355122}) asserts that a compact 6-dimensional positive monotone symplectic manifold endowed with a Hamiltonian $S^1$ action is always diffeomorphic to a Fano threefold. Even though the conjecture itself is still open, considerable progress has been made toward proving it and its generalizations \cite{MR3905680,MR3977283,MR4223523,MR4223522,MR4431178,MR5005925,MR5060441}.
\\$\;$

In this paper we make substantial progress towards proving that compact positive monotone symplectic manifolds endowed with some special Hamiltonian symmetries 
behave very similarly to Fano varieties. 
We give an informal overview of the main ideas, results and consequences of this paper, leaving precise statements and further details to Sections \ref{driving questions} and \ref{results}.
\begin{itemize}[leftmargin=*]
\item[-] We prove several \emph{finiteness results}, concerning the number of (equivariant) cobordism classes of compact positive monotone symplectic manifolds of a certain dimension and Euler characteristic, and endowed with
a torus action that is ``sufficiently generic'' (a so called GKM$_3$ action). This translates into the finiteness of certain \emph{reflexive GKM graphs} (as introduced in \cite[Section 5.3]{MR3695881}), which are a generalization of (smooth) reflexive polytopes
(see Theorem \ref{thm: finiteness intro} and Corollary \ref{cor reflexive graphs intro}). 
\item[-] In the setting above, using equivariant techniques and the geometry of numbers, we prove two kinds of \emph{boundedness results}: One regarding the \emph{diameter of the moment map} (see Theorem \ref{thm bd intro}), and the other regarding \emph{quantitative upper bounds for all the Chern numbers} (see Theorem \ref{theorem C}), which include an upper bound for the volume $c_1^n[M]$. 
The latter can be regarded as a positive monotone, symplectic
version of the Koll\'ar-Miyayoka-Mori volume upper bound.
\item[-] Our results apply to Fano varieties endowed with the Hamiltonian symmetries used in this paper. 
In particular, to the best of the authors' knowledge, \emph{the results on the bounds of the Chern numbers other than $c_1^n[M]$ are new for Fano varieties} as well.
\end{itemize}
Our finiteness results suggest that a complete classification of such manifolds is possible.
\begin{rem}
Our bounds and finiteness results depend explicitly on the Euler characteristic $\chi$ of $M$. 
However these bounds only depend on $n$, 
provided
that the even Betti numbers of $M$ are unimodal and the index of $(M,\omega)$ is large enough (see Propositions \ref{P:removing-chi}, \ref{upper bound chi k0} and Corollary \ref{C: dim 8 bounded cln}).
\end{rem}

\subsection{Driving questions}\label{driving questions}

Suppose that a compact torus $T$ acts on $(M,\omega)$ preserving the symplectic form.
Such action is called \textit{Hamiltonian} 
 if there exists a $T$-invariant map $\mu\colon M \to \mathfrak{t}^*$, called \textit{moment map}, such that 
\begin{equation}\label{def moment map}
d\langle \mu, \xi \rangle = - \iota_{\xi^\#} \omega \quad \text{for all  }\xi \in \mathfrak{t},
\end{equation}
where $\xi^\#$ is the vector field associated to $\xi$ and $\langle \cdot , \cdot \rangle$ is the natural pairing between $\mathfrak{t}^*$ and $\mathfrak{t}$. Unless otherwise stated, we assume that the torus action is effective and call the triple $(M,\omega,\mu)$ a \textbf{compact Hamiltonian $T$-space}. 

Given the (counter)example built by Fine and Panov \cite{MR2679581}, it is not hard
to produce examples of compact positive monotone Hamiltonian $T$-spaces that are not
simply connected, like Fano varieties are. 
Therefore in this paper we consider a class of
compact positive monotone Hamiltonian $T$-spaces that are automatically simply connected, namely those for which the action has isolated fixed points.

The following arise naturally: 

\vspace{0.5cm}
\textbf{Questions}:
\begin{itemize}
\item[(A)] For each $n\in \NN$, are there \emph{finitely many complex cobordism classes} of compact positive monotone Hamiltonian $T$-spaces $(M,\omega,\mu)$ with isolated fixed points and with $\dim(M)=2n$?
\item[(B)] Modulo $\GL(k;\ZZ)$ transformations, can the \emph{image of the moment map} $\mu(M)$ of any compact positive monotone Hamiltonian $T$-space
$(M,\omega,\mu)$ of a fixed dimension be \emph{embedded in a ``box''} in $\mathfrak{t}^*$ of some quantifiable size? 
\item[(C)] If the answer to (A) is affermative, are there \textit{quantitative upper bounds} on their Chern numbers, and therefore on their volume $\int_M \omega^n=c_1^n[M]$? 
\end{itemize}
Question (B) is motivated by 
\emph{smooth toric Fano varieties}, which admit an effective action of a complex torus $(\CC^*)^n$,
where $n$ is the complex dimension of the smooth toric variety. They are indeed positive
monotone Hamiltonian $T$-spaces $(M,\omega,\mu)$, where the compact torus $T$ in
$(\CC^*)^n$ has the maximal dimension possible, namely $\dim(T)=\frac{\dim(M)}{2}=n$. In
this case the image of the moment map $\mu(M)$ is --modulo translations-- a smooth
reflexive polytope (see \cite{MR1269718} or, for a symplectic proof, 
\cite[Prop.\ 1.8]{MR2507748}, \cite[Section 3]{MR2768658} or
\cite[Prop.\ 3.10]{MR3695881}). Therefore, as a consequence of a theorem of Lagarias and Ziegler
\cite{MR1138580}, modulo a $\GL(n;\ZZ)$ transformation, $\mu(M)$ can be embedded in a box (a
hypercube) in $\RR^n$ whose size just depends on $n$.  

Question (C) is inspired by the quantitative upper bound on the volume of a Fano variety found by Koll\'ar-Miyaoka-Mori;
however we are interested in finding bounds for all Chern numbers. 

Initial progress to answer the questions above was made in \cite{MR5005925},
\cite{MR5060441} and
\cite{MR4940256}, where equivariant topological invariants of low dimensional  --as well as tall
complexity one-- compact positive monotone Hamiltonian $T$-spaces are classified. Although
Questions (A), (B) and (C) are not addressed explicitly, their answers are positive and arise as a byproduct 
of the classification results presented therein.
\\$\;$

In this paper we answer Questions (A), (B) and (C) for a class of compact positive monotone Hamiltonian $T$-spaces called GKM$_3$,
which for $n>3$ are more general than smooth toric varieties. 
The GKM$_3$ condition can be stated in the following geometric terms:
For each $h\leq \dim(T)$, consider the $h$-skeleton $\mathcal{S}_h$ of $(M,\omega,\mu)$, namely the union of the orbits of dimension at most $h$. This is a union of compact isotropy submanifolds, and therefore a union of compact Hamiltonian submanifolds $N^h$, where the (quotient) torus acting effectively on each $N^h$ has dimension $h$. A \textbf{compact Hamiltonian GKM$_k$ space}
is a compact Hamiltonian $T$-space, such that the action of the (quotient) torus on each $N^h$ is \emph{toric}, for each $h< k$ (in particular the condition for $h=0$ is equivalent to having discrete fixed points; see Section \ref{GKMk} and in particular Remark \ref{gkmk ham}). 
In this case we refer to $\mathcal{S}_h$ as the \emph{toric $h$-skeleton of} $(M,\omega,\mu)$, for each $h<k$. 

GKM$_2$ spaces were first introduced by Goresky, Kottwitz and MacPherson in their seminal
paper \cite{MR1489894} and are known in the literature simply as GKM spaces. We note that
symplectic toric manifolds, namely compact Hamiltonian $T$-spaces with
$\dim(T)=\dim(M)/2=n$, are GKM$_n$ spaces. 
Many compact positive monotone Hamiltonian $T$-spaces, for instance generalized flag varieties, admit GKM$_k$ actions for some $k\geq 3$, but no toric action.
For instance 
the Grassmannian of complex planes in $\CC^4$ is a GKM$_3$ space with respect to the natural action of a compact torus of dimension 3 (see \cite[Section 5]{MR3695881}).

\begin{rem}
We would like to remark that, even though the examples of compact, positive monotone Hamiltonian GKM$_3$ spaces presented in this paper admit a K\"ahler structure that
turns them into Fano varieties, an example of non-K\"ahlerizable Hamiltonian $T$-space with isolated fixed points is not known and it is a long standing open problem in equivariant symplectic geometry to find such an example.  
\end{rem}

\subsection{Results}\label{results} The first result of this paper concerns certain integers that can be associated to a compact (Hamiltonian) GKM space $(M,\omega,\mu)$ and are defined as follows. 
Since the toric 1-skeleton of $(M,\omega,\mu)$ is a union of invariant symplectic spheres $S^2$, the tangent bundle $TM$ splits (equivariantly but not uniquely) over each one of these as a direct sum of complex line bundles $\mathbb{L}_i$, $i=1,\ldots,n$. For each $S^2$, consider the integers $c_1(\mathbb{L}_i)[S^2]$, $i=1,\ldots,n$, given by evaluating the first Chern class of $\mathbb{L}_i$ over $S^2$.  
We refer to these integers as the \textbf{line Chern numbers} of the compact (Hamiltonian) GKM space $(M,\omega,\mu)$ corresponding to a chosen equivariant splitting of $TM$ over each invariant
sphere. 

If $(M,\omega,\mu)$ is a compact Hamiltonian GKM$_3$ space, then the equivariant splittings described above are unique. Using convexity properties of the moment map image of the toric 2-skeleton, we prove

\begin{thm}[see Theorem \ref{thm: mainthm} and Remark \ref{rem open close toric}]\label{thm: mainthmintro}
Let $n,\chi\in \NN$. Then there are only finitely many possible values for the line Chern numbers of a compact 
positive monotone
Hamiltonian GKM$_3$ space $(M,\omega,\mu)$ of dimension $2n$ and with Euler characteristic $\chi$.
\end{thm}
In fact we obtain concrete quantitative bounds on the line Chern numbers from above and below (see Corollary \ref{global bound}). It is important to note that this step, on which many subsequent results depend, relies on both the assumptions that the action is Hamiltonian and that it is GKM$_3$.
We provide counterexamples if either condition is weakened in Examples \ref{line chern numbers not Ham} and \ref{ex: long weights}.

Equipped with bounds on the line Chern numbers, we proceed to investigate their quantitative implications for the sizes of the weights and the Chern numbers. 
This analysis is carried out at the level of abstract GKM graphs, independently of the preceding geometric setup.
In particular GKM graphs are neither required to come from Hamiltonian actions nor does the GKM$_3$-condition play a role. To state our main results in this direction we recall the concept of a 
\textit{maximal extension of a GKM graph}, which was first introduced by Kuroki (see the original reference
\cite{MR3943448} as well as Section \ref{subsec: max extensions} for our notation).
Loosely speaking, a maximal extension of a GKM graph is the graph itself endowed with labels in the largest lattice possible, while keeping the original connection
and line Chern numbers fixed.
 We also call GKM graphs (resp. actions) of this type \emph{combinatorially maximal} (see Definitions \ref{def comb maximal} and \ref{def combinatorially maximal}). 
 Geometrically, an extension of this type corresponds to an extension of the acting torus. 
 However, while an action being combinatorially maximal ensures that it cannot be geometrically extended effectively with the given connection, the converse is not known.

\begin{thm}\label{thm: boxintro}
Let $(\Gamma,\alpha,\nabla)$ be an $n$-valent GKM graph with fixed connection $\nabla$, axial function $\alpha$ and $\chi$ vertices, and assume that the absolute value of
all line Chern numbers is bounded above by $\beta>0$.
\begin{enumerate}
\item[(i)]
There exists a constant $K=K(\beta,n,\chi)$, explicitly given in \eqref{def K}, depending only on $\beta$,
$n$, and $\chi$, and a choice of maximal extension $\tilde\alpha\colon E\rightarrow \ZZ^k$
of $(\Gamma,\alpha,\nabla)$, for which 
\[
\| \tilde\alpha(e) \|_\infty\leq K\quad \text{for all  } e\in E\,.
\]
\item[(ii)] There exists a constant $L=L(\beta,n,\chi)$, explicitly given in \eqref{eq def L}, such that for each decomposition $n=i_1+\ldots+i_m$, the corresponding combinatorial Chern number of $(\Gamma,\alpha)$ satisfies
\[|c_{i_1,\ldots,i_m}[\Gamma]|\leq L.\]
\end{enumerate}
\end{thm}

The proofs of $(i)$ and $(ii)$ are given in Theorems \ref{thm: Box} and Corollary \ref{thm: quantitative bound}.
To understand the role of the maximal extension in $(i)$, we point out that, in general, bounds on the line Chern numbers do not imply any reasonable restrictions on the weights that can occur on a GKM manifold. For instance, the various possible restrictions of an action to subtori can produce weights of arbitrary length, even up to automorphism of the weight lattice (cf. Example \ref{ex: long weights}). 
Hence the passage to a maximal extension is necessary for a bound on the size of the weights to exist. To prove $(i)$, we use, inter alias, methods coming from the geometry of numbers. More precisely we use Siegel's Lemma and its improved version due to Bombieri and Vaaler,
together with the theory of lattice reduction
(in particular Korkine-Zolotarev bases). The essential ingredient for the proof of part $(ii)$ is the bound on the size of the labels obtained in $(i)$. We point out that, in the case where the GKM graph comes from a GKM manifold, the Chern numbers in $(ii)$ agree with the respective Chern numbers of the manifold. So the bounds in $(ii)$ apply immediately to Chern numbers of GKM manifolds.

$\;$

Returning to the original motivation of studying questions (A), (B), and (C) for positive monotone GKM$_3$ manifolds our results are the following.
\begin{maintheorem}[\textbf{Finiteness}, see Theorem \ref{thm: bordism} (1)]\label{thm: finiteness intro}
For every $\chi,n\in \NN$, there are finitely many complex cobordism classes of compact positive monotone symplectic manifolds of dimension $2n$ and Euler characteristic $\chi$ admitting a Hamiltonian GKM$_3$ action. 
\end{maintheorem}

The strategy of the proof is to use Theorem \ref{thm: mainthmintro} to deduce that there are finitely many possible
\emph{maximal extensions} of the GKM graphs associated to compact positive monotone Hamiltonian GKM$_3$
spaces of a given dimension and Euler characteristic (Corollaries \ref{cor: main} and \ref{virtual iso and max extensions}). 
In order to obtain the theorem, it is then enough to prove that a maximal extension of a GKM graph determines the Chern numbers of the underlying symplectic manifold.

The obvious equivariant version of the above theorem is not directly true, as finiteness of equivariant cobordism classes does not behave well with respect to restricting the action to subgroups (see Example \ref{ex: eqbordism}). If we rule out this phenomenon by imposing that the action is combinatorially maximal,
 we indeed obtain an equivariant analogue of Theorem \ref{thm: finiteness intro} (see Theorem \ref{thm: bordism} (2)).
As a byproduct of this we can also prove the following finiteness result, which mirrors the analogue for reflexive polytopes (see Section \ref{reflexive GKM graphs}).

\begin{cor}\label{cor reflexive graphs intro}
For each $n,\chi\in\NN$, there are finitely many reflexive GKM$_3$ graphs with $\chi$ vertices that are $n$-valent and combinatorially maximal.
\end{cor}

\vspace{0.3cm}

The next main result of this paper is of quantitative nature and concerns the \emph{diameter of the moment map} $\mu\colon M \to \mathfrak{t}^*\cong \RR^r$, defined as 
\begin{equation}\label{diameter}
d_\mu:=\max\{\|\mu(p)-\mu(q)\|_\infty, \;p,q\in M\}\,.
\end{equation}

\begin{maintheorem}[\textbf{Boundedness of the moment map}, see Corollary \ref{cor: box} and Remark \ref{rem better bound}]\label{thm bd intro}

Let $(M,\omega,\mu)$ be a compact positive monotone Hamiltonian $\text{GKM}_3$ space of dimension $2n$ and Euler characteristic $\chi$.
Assume that the action is combinatorially maximal and that $\dim(T)=r$.

Then there exists a constant $K'=K'(n,\chi)$ such that, modulo $\GL(r,\ZZ)$ transformations on $\mft^*$,
\begin{equation}\label{diameter intro}
d_\mu \leq K'(n,\chi)\,.
\end{equation}
Moreover for every $n\geq 3$ the constant $K'$ satisfies 
 \[K'(n,\chi)\leq \frac{n^2(n+3)\chi}{4}\cdot\left(\frac{ n^4\chi}{2}\right)^{n\chi}\,.\]
\end{maintheorem}

 Just like equivariant cobordism classes, the diameter $d_\mu$ does not behave well with respect to restricting the action to subgroups (see Example \ref{ex: long weights}). This makes the combinatorial maximality
 condition in the theorem necessary. 
The proof is achieved using Theorems \ref{thm: mainthmintro} and \ref{thm: boxintro} $(i)$.
Finally, as a consequence of Theorems \ref{thm: mainthmintro} and \ref{thm: boxintro} $(ii)$ we obtain
\begin{maintheorem}[\textbf{Bounds on the Chern numbers}, see Corollaries \ref{cor: quantitativeboundgeometric} and \ref{bound on the volume}]\label{theorem C}
Let $(M,\omega,\mu)$ be a compact positive monotone Hamiltonian $\text{GKM}_3$ space of dimension $2n$ and Euler characteristic $\chi$.

Then for each decomposition $n=i_1+\ldots+i_m$, there is a constant $L=L(n,\chi)$ depending only on $n$ and $\chi$, explicitly given in \eqref{eq def L geometric}, such that the corresponding Chern number satisfies
\[|c_{i_1,\ldots,i_m}[M]|\leq L.\]
Moreover there exists a polynomial $P(n,\chi)$ such that $L$ satisfies
\[L \leq P(n,\chi)^{n^3\chi}.\]

\end{maintheorem}

Note that in particular we obtain an explicit quantitative bound on \[\int_M \omega^n=c_1^n[M],\]
see Corollary \ref{bound on the volume}.

Last we derive some consequences of our results for smooth reflexive polytopes.
For instance, as a direct application of Theorem \ref{thm bd intro}, we have the following

\begin{cor}\label{cor reflexive1}
There exists a polynomial $P(x,y)$ such that, modulo $GL(n,\ZZ)$ transformations, for every smooth reflexive polytope $\Delta$ of dimension $n\geq 3$ with $|V|$ vertices, we have

\begin{equation}\label{diameter refl}
d_\Delta:=\max\{\|p-q\|_\infty, \;p,q\in \Delta\}\leq \frac{n^2(n+3)|V|}{4}\cdot\left(\frac{n^4|V|}{2}\right)^{n|V|}\,.
\end{equation}

\end{cor}

The bounds on the line Chern numbers exhibited in Theorem \ref{thm: mainthm} can be translated into the boundedness of some integers associated
to the two dimensional faces of a smooth reflexive polytope (see Corollaries \ref{cor reflexive} and \ref{global bound reflexive}). This, in turn, can be applied to estimate
the reflexive dimension of a smooth polytope (see Remark \ref{reflexive dimension}).

$\;$\\

\noindent\textbf{Acknowledgments} 
The first author was supported by the Deutsche Forschungsgemeinschaft under the project
number 452427095.
The second author thanks Alessio Corti for the kind
hospitality at Imperial College and for helpful discussions regarding Questions (G) and
(H) in Section \ref{open questions}. 
She also thanks Christian Haase, Johannes Hofscheier and Benjamin Nill for valuable
 insights into Section \ref{reflexive GKM graphs}, and Dmitri Panov for discussions about Remark 
 \ref{fine panov}.
The third author wants to thank Alberto Espuny D\'iaz for helpful discussions.
 The second and third authors were partially supported by SFB/TRR 191 grant 
 ``Symplectic Structures in Geometry, Algebra and Dynamics'' funded by
the Deutsche Forschungsgemeinschaft.  The third author is furthermore supported by the grant PID2024-155646NB-I00 of the Agencia Estatal de Investigaci{\'o}n.
Finally, the authors would also like to thank Susan Tolman for discussing ideas, consequences and perspectives of their results.

\section{Preliminaries}
\subsection{GKM theory}\label{subsection GKM theory}
In their seminal work \cite{MR1489894}, Goresky, Kottwitz and MacPherson introduced 
a certain remarkable class of torus actions on manifolds, defined as follows. 

\begin{defn}\label{D: GKM manifold}
Let $T=T^{r}$ denote a compact torus of rank $r\geq 2$ and $M$ a compact, connected and orientable manifold of
dimension $2n$. If $T$ acts on $M$ effectively such that
\begin{enumerate}[label=(\alph*)]
	\item there is an almost complex structure $J$ on $M$ invariant under $T$, 
	\item $H^{\mathrm{odd}}(M;\mathbb{Z}) = 0$
	\item the set of fixed points
		\[
			M^{T} := \{p \in M \colon  T \cdot p = \{p\}\} 
		\]
		is finite and
	\item the \emph{1-skeleton} $$\mathcal{S}_{1} := \{ p \in M \colon \dim(T \cdot p)
		\leq 1\}$$ is a finite union of $T$-invariant $2$-spheres
\end{enumerate}
then $(M,J,T)$ is called a \textbf{GKM space} and the action of $T$ on $M$ is called a
\emph{GKM action}. 
\end{defn}
\begin{rem}
For a condition equivalent to (d) see Section \ref{GKMk}.
\end{rem}

\begin{defn}\label{def Ham GKM}
Let $(M,\omega,\mu)$ be a compact Hamiltonian $T$-space, where the set of fixed points $M^T$ is finite. We say that 
$(M,\omega,\mu)$ is a \textbf{compact Hamiltonian GKM space} if condition (d) above holds. 
\end{defn}

\begin{rem}\label{rem: contractible set of acs}
We note that the
space of almost complex structures compatible with a given symplectic form is non-empty and
contractible (see  \cite[Proposition 4.1.1]{MR3674984}). 
Moreover, there always exists a compatible
almost complex structure which is invariant under the $T$-action, see \cite[Lemma
5.5.6]{MR3674984}. Also observe that standard Morse theoretical arguments imply that
condition (b) holds for a compact Hamiltonian $T$-space with isolated fixed points.
It follows that a compact Hamiltonian GKM space is also a GKM space as in Definition \ref{D: GKM manifold}.
\end{rem}

At a fixed point $p \in M^{T}$ we have an isotropy representation of $T$ on the complex
vector space $T_{p}M$ and thus also a representation of the Lie algebra $\mathfrak t$ of
$T$ by taking derivatives. This complex representation decomposes the tangent space into
complex $1$-dimensional subspaces
\begin{equation}\label{V alpha}
	T_{p}M = \bigoplus_{\alpha} V_{\alpha}
\end{equation}
where the corresponding $\alpha \in \mathfrak t^{\ast}$ are called the \emph{weights} of the isotropy representation of $T$ at $p$
and
\[
	V_{\alpha} = \{ v \in V : X * v = \mathrm{i}\alpha(X) \cdot v, \quad X \in \mathfrak
	t\}.
\]
The \emph{weight lattice} of $T$ is defined to be
$\mathbb{Z}_{\mathfrak t} = \ker \exp$, where $\exp\colon \mathfrak{t}\to T$ is the exponential map, and $\alpha \in
\operatorname{Hom}(\mathbb{Z}_{\mathfrak t}, \mathbb{Z}) = \mathbb{Z}_{\mathfrak
t}^{\ast}$.

A GKM space defines naturally a combinatorial object, the \emph{GKM graph}.
We recall the necessary facts here and refer the reader to
\cite{MR1823050} for a detailed exposition. Condition (d) in Definition \ref{D: GKM
manifold} asserts that $\mathcal{S}_1$ is a union of invariant spheres, therefore
the quotient $\mathcal{S}_1/T$ is homeomorphic to a graph, where the vertices are the fixed points
$M^{T}$
 and two vertices are connected by an edge if the two corresponding fixed points lie on the same
invariant $2$-sphere. Thus
geometrically, every (undirected) edge corresponds to one invariant sphere. However, in abstract GKM theory 
it is often assumed that to each sphere with fixed points $p$ and $q$, one associates 
two directed edges, namely $e=(p,q)$ and $\overline{e}=(q,p)$, both representing geometrically the same sphere. For $e=(p,q)$ we say that 
$p$ is the \emph{initial point} of $e$ and we denote it by $i(e)$, and $q$ is the \emph{terminal point}, and denote it by $t(e)$. 
Let $E$ be the set of directed edges of the GKM graph $\Gamma$ and 
$E_p$ be the set of directed edges with initial point given by the vertex $p$.

Each directed edge carries a weight $\alpha\in \mathbb{Z}_{\mathfrak
t}^{\ast}$, namely for every directed $e=(p,q)$ we associate the weight $\alpha$ of the isotropy representation of
$T$ on $T_pS_e^2$, where $S_e^2$ is the sphere associated to the (undirected) edge $e$.
The \emph{axial function} is the function
that assigns to each directed edge $e$ the weight defined above. With abuse of notation we denote such a function
by $\alpha\colon E \to \mathbb{Z}_{\mathfrak
t}^{\ast}$. 

The axial function satisfies some important ``compatibility properties'': For each $e\in E$, consider the corresponding sphere $S^2_e$.
The restriction of the tangent bundle $TM|_{S^2_e}$ splits
equivariantly into a sum of equivariant line bundles $\mathbb{L}_1,\ldots,\mathbb{L}_n$.
We also observe that, for each $i=1,\ldots,n$, $\mathbb{L}_i|_p$ (resp.\ $\mathbb{L}_i|_q$) corresponds to the tangent space of a sphere $S^2_f$,
for some $f\in E_p$ (resp.\ for some $\widetilde{f}\in E_q$). 
We obtain a bijection $\nabla_e\colon E_p \to E_q$ sending $f$ to
$\widetilde{f}$ which we refer to as a \emph{connection} along $e=(p,q)\in E$.

We note that, for every $e\in E$
the following properties hold:
\begin{itemize}
\item[(a)] $\nabla_e(e)=\overline{e}$, where $e=(p,q)$ and $\overline{e}=(q,p)$,
\item[(b)] $(\nabla_e)^{-1}=\nabla_{\overline{e}}$, 
\item[(c)] $\alpha(\overline{e})=-\alpha(e)$ and
\item[(d)] 
\begin{equation}\label{def ce(f)}
			\alpha(f) - \alpha(\nabla_{e}f) = c_{e}(f) \alpha(e). 
		\end{equation}
		for some integer $c_e(f)$, which we refer to as \emph{line Chern number}.
\end{itemize}
While properties (a), (b) and (c) come straightforward from the definition of connection and axial function as well as the GKM condition, property (d) comes from observing
that the subgroup $$K_e:=\exp(\{\xi\in \mathfrak{t}\mid \alpha(\xi)\in \ZZ\})$$ acts trivially on $S^2_e$, and therefore the isotropy
representations of $K_e$ on $\mathbb{L}_i|_p$ (corresponding to the edge $f$) and $\mathbb{L}_i|_q$ (corresponding
to the edge $\widetilde{f}=\nabla_e(f)$) should be the same. Condition \eqref{def ce(f)} is called \emph{compatibility} of the connection $\nabla$ with the
axial function $\alpha$. 

These notions lead to the definition of an \emph{integral abstract GKM graph}.

\begin{defn}\label{D: GKM graph}
Suppose $\Gamma$ is an $n$-valent graph without loops. We denote by $V$ the set of
vertices and by $E$ the set of directed edges of $\Gamma$, which must satisfy the following condition: a directed edge $e=(p,q)\in V\times V$
is in $E$ if and only if $\overline{e}:=(q,p)$ is.
Let $i(e)$ be the initial vertex of $e$ and $t(e)$ its terminal
vertex. For $p \in V$ let $E_{p}$ denote the subset of $E$ with $i(e)=p$. 
A \emph{connection} $\nabla$ on $\Gamma$ is a family of bijections
\[
	\nabla_{e} \colon E_{i(e)} \longrightarrow E_{t(e)}
\]
such that, for all $e \in E$, properties (a) and (b) above hold.

An \textbf{integral abstract GKM graph} is a pair $(\Gamma, \alpha)$ where $\Gamma$ is a
graph as above and $\alpha \colon E \to \mathbb{Z}^{r}$ a map, called \emph{(integral abstract) axial function}, 
with the following properties
\begin{itemize}
	\item[(i)] For every $p \in V$ and every two edges $e,f \in E_{p}$ with $e \neq f$,
		the elements $\alpha(e)$ and $\alpha(f)$ are linearly independent.
	\item[(ii)] $\alpha(\overline{e})=-\alpha(e)$ for all $e \in E$.
	\item[(iii)] There exists a connection $\nabla$ which is \emph{compatible} with $\alpha$, namely
	condition (d) above holds for every $e\in E$. 
		We call the number $c_{e}(f)$ the \emph{abstract line Chern number}.
\end{itemize}
\end{defn}
We denote the integral abstract GKM graph by the triple $(\Gamma,\alpha,\nabla)$ if the choice of
connection is relevant for our purposes and $(\Gamma,\alpha)$ if not. 

Note that in the literature \emph{rational abstract GKM graphs} have also been defined, where in this case the axial function 
takes values
in $\mathbb{Q}^r$, 
satisfies (i) and (ii) above, and there exists a compatible connection satisfying (d), where however the corresponding 
$c_e(f)$ are rational numbers. 
In this article we work solely with integer abstract GKM graphs, henceforth simply
called \textbf{abstract GKM graphs}. 

\begin{rem}\label{R: Existence of GKM graph}
It is clear that every GKM graph coming from a GKM space $(M,J,T)$ is an abstract GKM graph. 
Indeed, if we fix an identification of $\mathbb{Z}_{\mathfrak{t}}^*$
with $\ZZ^r$, then conditions (i), (ii) and (iii) follow
from the discussion before Definition \ref{D: GKM graph}.

We also note that, for a Hamiltonian GKM space $(M,\omega, \mu)$, the GKM graph $(\Gamma, \alpha)$
can be seen in $\mathfrak{t}^{\ast}$ using the moment map 
$\mu$: Each sphere $S^2_e$ is mapped to $\mu(S^2_e)$, for every $e\in E$.
 Observe that in this case, since each $S^2_e$ is a symplectic submanifold of $(M,\omega)$ with a (non effective) Hamiltonian action of $T$, the weight $\alpha(e)$
and the vector $\mu(i(e))-\mu(t(e))$ have the same slope.
\end{rem}

\begin{rem}\label{close open}
Suppose that $(\Gamma, \alpha)$ is the GKM graph of a Hamiltonian GKM space $(M,\omega, \mu)$,
where $\mu$ is injective on the fixed point set $M^T$. Then, using Remark \ref{R: Existence of GKM graph} 
and seeing the GKM graph as a subset of $\mathfrak{t}^*$, we can give the line Chern
numbers associated to a connection $\nabla$ the following
geometric interpretation: 
By \eqref{def ce(f)} the line Chern numbers give a measure of whether the image through $\mu$ of
the edges $f$ and $\nabla_e(f)$ are parallel ($c_e(f)=0$), whether they ``close up'' ($c_e(f)>0$) or they ``open up'' ($c_e(f)<0$).
\end{rem}
\begin{figure}[h]
	\centering
	\includegraphics{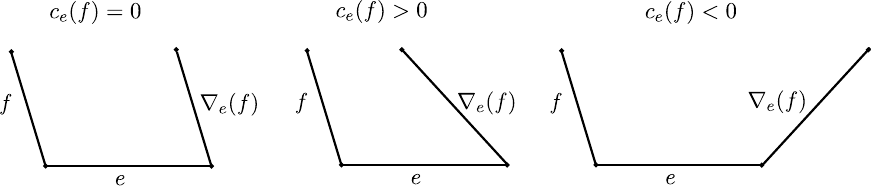}
	\label{line chern numbers}
\end{figure}

\begin{rem}\label{line chern triples}
Given an abstract GKM graph $(\Gamma,\alpha,\nabla)$, by the properties of the axial function and its compatibility with the connection, it can be easily seen that
\begin{align}
\label{alpha e e} c_e(e)=& c_{\overline{e}}(\overline{e})=2\,\quad\text{for all  }e\in
E,\;\;\text{and}\\
\label{eq cefs} c_e(f)=& c_{\overline{e}}(\nabla_e(f))\,, \quad\text{for all  }f\in
E_{i(e)}, \;e\in E\,.
\end{align}
By \eqref{eq cefs}, the integer $c_e(f)$ should be thought of as being associated to the
\emph{unoriented} edge $e$ together with the \emph{pair} $\{f,\nabla_e(f)\}$. In Lemma \ref{geom mean line numbers} we give indeed
a geometric interpretation of these integers which explains this combinatorial fact.
\end{rem}

Let $(M,J,T)$ be a GKM space.
We note that the equivariant splitting of $TM|_{S^2_e}$ into equivariant line bundles is in
general not unique. Therefore there may be different connections associated to different splittings, which
 are compatible with the same axial function $\alpha$,
but have different line Chern numbers (see Example \ref{ex: weird cube}). 

On the other hand it is reasonable to ask whether, given a GKM space with GKM graph $(\Gamma,\alpha)$, an edge $e\in E$, and an abstract 
bijection $\nabla_e\colon E_{i(e)}\to E_{t(e)}$ compatible with $\alpha$, 
there exists a splitting of $TM|_{S^2_e}$ inducing the bijection $\nabla_e$. This is
answered in the following lemma, which is not needed for the proofs of
our main results but is relevant in the context of the combinatoric counterexamples given
in this paper.

\begin{lem}\label{geom and comb connection}
Let $(M,J,T)$ be a GKM space of dimension $2n$ with GKM graph $(\Gamma,\alpha)$. Assume that, for $e\in E$, there exists a bijection $\widetilde{\nabla}_e\colon E_{i(e)}\to E_{t(e)}$ compatible with
the given axial function $\alpha$. 
Then there exists a splitting of $TM|_{S^2_e}$ such that
the induced (geometric) connection $\nabla_e$ along $e$ is exactly $\widetilde{\nabla}_e$.
\end{lem}
We may summarize this lemma by saying that on the GKM graph of a GKM space, ``each combinatorial connection compatible
with the axial function is a geometric connection coming from a splitting''. 

\begin{proof}
Let $e$ be an edge in the GKM graph $(\Gamma,\alpha)$ with $p:=i(e)$, $q:=t(e)$, and with
$\gamma=\alpha(e)$. Let $S^2_e$ be the $2$-sphere associated to $e$. Then, as described
above, there is an equivariant splitting $TM|_{S_e^2}=\bigoplus \mathbb L_i$ into complex line
bundles which induces the ``geometric connection'' $\nabla_ef_i=\widetilde{f}_i$, where $f_i\in E_{i(e)}$
represents $(\mathbb L_i)_p$ and $\widetilde{f}_i$ represents $(\mathbb L_i)_q$. Moreover, if
the weight of the $T$ isotropy representation on $(\mathbb L_i)_p$ is $\alpha_{i,p}$
 and that on $(\mathbb L_i)_q$ is $\alpha_{i,q}$, we have that $\alpha_{i,p}$ and $\alpha_{i,q}$
 agree up to integral multiples of $\gamma$, for all $i=1,\ldots,n$.

Now assume that also $\alpha_{i,p}$ and $\alpha_{j,q}$
agree up to integral multiples of $\gamma$ for some $1\leq i\neq j\leq n$. This in fact
holds pairwise for all four of  $\alpha_{i,p},\alpha_{j,p},\alpha_{i,q},\alpha_{j,q}$ and
combinatorially we can modify $\nabla_e$ by sending $f_i$ to $\widetilde{f}_j$ and $f_j$
to $\widetilde{f}_i$; this new connection is still compatible with the given axial function $\alpha$ for the assumption
made above.
It is clear that every connection compatible with $\alpha$ arises from
$\nabla$ by iterating operations of this type. Hence, it remains only to prove that in this
case the modified connection is geometrically justified, i.e. we need to find 
line bundles $\mathbb K$, $\mathbb K'$ over $S^2_e$
such that 
$\mathbb L_i\oplus \mathbb L_j = \mathbb K\oplus \mathbb K'$ where $\mathbb
K_p=(\mathbb L_i)_p$, $\mathbb K_q=(\mathbb L_j)_q$ and $\mathbb K_p'=(\mathbb L_j)_p$,
$\mathbb K_q'=(\mathbb L_i)_q$.

Let $V_k=(\mathbb L_k)_p$, $k=i,j$, be the respective irreducible summands of $T_p M$ and
let $D_p\subset S^2_e$ be the hemisphere containing $p$. As $D_p$ deformation retracts
equivariantly onto $p$ we see that $\mathbb L_k|_{D_p}\cong D_p\times V_k$ with the diagonal
action. We restrict to a $T$-invariant annulus $S^1\times [0,1]\cong A\subset D_p$ where
$T$ acts on the $S^1$ component through the action defined by $\gamma$. It suffices to
construct a splitting $A\times (V_i\oplus V_j)=\mathbb K\oplus \mathbb K'$ into
$T$-equivariant line bundles (over the annulus) such that the decomposition agrees with the decomposition
$V_i\oplus V_j$ over points in $\partial A$ but the factors are swapped along a path
between the connected components of $\partial A$.

By assumption, the weights of $V_i,V_j$ satisfy $\alpha_{i,p} = \alpha_{j,p}+k\gamma$ for some $k\in \ZZ$. Then the map
\begin{align*}
S^1\times [0,1]\times V_i&\rightarrow S^1\times [0,1]\times V_j\\
(z,h,v)&\mapsto (z,h, z^{k}\cdot v)
\end{align*}
(where multiplication in the last component is just complex multiplication) is a
$T$-equivariant isomorphism of complex $T$-vector bundles over $A$. In particular $A\times
(V_i\oplus V_j)\cong A\times (V_i\oplus V_i)$ and it suffices to construct $T$-equivariant
splittings of the latter. Since the action on $V_i\oplus V_i$ is just through diagonal
complex multiplication, every (nonequivariant) splitting of $V_i\oplus V_i=W\oplus W'$
into complex lines automatically extends to an equivariant splitting $S^1\times
\{h\}\times (W\oplus W')$ over every $T$-invariant circle in $A$. Now let $\eta$ be a path
in $\GL(2;\CC)$ from the identity to the automorphism which swaps the standard basis. Then
$\eta(t)$ gives a parametrized family of splittings of $\CC^2\cong V_i\oplus V_i$. Using
the $[0,1]$ component of $A=S^1\times [0,1]$ as parameter for $\eta$, we obtain the
desired equivariant splitting of $A\times (V_i\oplus V_i)$.

\end{proof}

We conclude this subsection with the notion of isomorphism of GKM graphs

\begin{defn}\label{D: Iso of GKM graphs}
Suppose $(\Gamma, \alpha)$ and $(\Gamma', \alpha')$ are two GKM graphs with axial
functions mapping into $\mathbb{Z}^{r}$. An \emph{isomorphism}
$\Phi \colon (\Gamma, \alpha) \to (\Gamma', \alpha')$ consists of an isomorphism of graphs
$\Phi \colon \Gamma \to \Gamma'$ together with an automorphism $\Psi \in \mathrm{GL}(r,
\mathbb{Z})$, such 
\[
	\alpha(\Phi(e)) = \Psi(\alpha(e)) \quad \text{for all } e\in E\,.
\]
\end{defn}

\subsection{Equivariant cohomology and equivariant graph cohomology}
A remarkable feature of GKM spaces is that the GKM graph encodes the equivariant
cohomology of the torus action. Let $G$ be a topological group, $EG \to BG$ the associated universal principal bundle
and suppose that $G$ acts continuously on a topological space $X$. Then
the \emph{Borel construction} is defined by
\[
	X_{G} := EG \times_{G} X
\]
and the \emph{equivariant cohomology} of $X$ is set to be
\[
	H_{G}^{\ast}(X;\Lambda) := H^{\ast}(X_{G};\Lambda) 
\]
where $\Lambda$ is a coefficient ring with unit and we take singular cohomlogy on the
right hand side. Now, suppose $G$ is a torus $T$ and $X$ is a manifold $M$. Then, since $T
\cong S^{1} \times \ldots S^{1}$ and $BS^{1} \cong \mathbb C\mathbb P^{\infty}$, we have
\[
	BT \cong \mathbb C\mathbb P^{\infty} \times \ldots \times \mathbb C\mathbb P^{\infty}.
\]
Recall that $H^{\ast}(\mathbb C\mathbb P^{\infty}; \Lambda)$ is the polynomial ring
$\Lambda[x]$ in one variable (where $x \in H^{2}(\mathbb C\mathbb P^{\infty}; \Lambda)
\cong \Lambda$ is a generator), hence $H^{\ast}(BT; \Lambda) \cong \Lambda[x_{1}, \ldots,
x_{r}]$ ($r = \mathrm{rank}\, T$). The projection $\pi\colon M_{T} \to BT$ is a fiber bundle with fiber
$M$ and $\pi^*$ gives $H^{\ast}_{T}(M;\Lambda)$ the structure of an $H^{\ast}(BT;\Lambda)$
module. If $p \in M$ is a fixed point, we have an equivariant embedding $\iota \colon
\{p\} \to M$ and therefore an induced map $\iota^{\ast} \colon H^{\ast}_{T}(M;\Lambda) \to
H^{\ast}_{T}(\{p\};\Lambda) \cong H^{\ast}(BT;\Lambda)$. If $x \in
H_{T}^{\ast}(M;\Lambda)$ we denote $\iota^{\ast}(x)$ by $x(p)$.

Suppose $E \to M$ is a $T$-equivariant complex vector bundle over a manifold $M$.
Applying the Borel construction on the total space and on the base, one obtains the complex vector
bundle $E_{T} \to M_{T}$. The $i$-th \emph{equivariant Chern class $c_{i}^{T}(E)$ of $E$},
is defined to be the ordinary $i$-th Chern class of $E_{T} \to M_{T}$, thus $c_{i}^{T}(E) :=
c_{i}(E_{T}) \in H^{\ast}(M_{T};\mathbb{Z}) = H_{T}^{\ast}(M;\mathbb{Z})$.

\begin{rem}\label{R: weights and cohomology}
For a GKM space $(M,J,T)$ the weights $\alpha$ are elements of $\mathbb{Z}_{\mathfrak
t}^{\ast}$. Such elements determine a linear form in $\mathfrak t^{\ast}$ which in turn
determines a \emph{character} $\lambda \in \operatorname{Hom}(T, S^{1})$. The induced map
on the level of classifying spaces $B \lambda \colon BT \to BS^{1} = \mathbb C\mathbb
P^{\infty}$ can be used to pull back the universal line bundle $\mathbb{L} \to BS^{1}$ to a line
bundle $\mathbb{L}_{\alpha} \to BT$. The map
\[
	\mathbb{Z}_{\mathfrak t}^{\ast} \to H^{2}(BT;\mathbb{Z}), \quad \alpha \mapsto
	c_{1}(\mathbb{L}_{\alpha})
\]
is an isomorphism. In that way we identify weights
with elements in $H^{2}(BT;\mathbb{Z})$. 
\end{rem}

Now we define the so called \emph{equivariant graph cohomology}
of an abstract GKM graph which, under certain assumptions, agrees with the
equivariant cohomology of the manifold, in the case in which the GKM graph
comes from a GKM manifold. 
 
\begin{defn}\label{D: graph cohomology}
Let $T$ be a torus of rank $r$ and choose an identification of $T$ with $S^{1} \times
\ldots \times S^{1}$, which leads to an identification of $\mathbb{Z}_{\mathfrak t}^{\ast}$ with $\mathbb{Z}^{r}$. Let $(\Gamma, \alpha)$ be an abstract GKM graph with axial function
$\alpha \colon E \to \mathbb{Z}^{r}$.  Using Remark \ref{R: weights and
cohomology} we can consider the axial function to be a map $\alpha \colon E \to
H^{2}(BT;\mathbb{Z})$.

For a coefficient ring $\Lambda$ with unit we set
the \emph{equivariant graph cohomology} of $(\Gamma, \alpha)$ to be
\[
	\mathcal{H}_{T}^{\ast}(\Gamma, \alpha; \Lambda) := \left\{ f\in \text{Maps}\colon V \to 
	H^{\ast}(BT;\Lambda) \, | \, f(i(e)) - f(t(e)) \equiv 0 \mod \alpha(e),\, \forall e \in
E \right\} 
\]
The compatibility condition $f(i(e)) - f(t(e)) \equiv 0 \mod \alpha(e)$ in the definition means that there exists $g_e\in H^{\ast}(BT;\Lambda)$,
such that $f(i(e)) - f(t(e))=g_e\cdot \alpha(e)$

The algebra $\mathcal{H}_{T}^{\ast}(\Gamma, \alpha; \Lambda)$ inherits a grading from
$H^*(BT;\Lambda)$ where an element $f \in \mathcal{H}_{T}^{\ast}(\Gamma, \alpha; \Lambda)$
has degree $2i$ if $f(v)$ is a homogeneous polynomial of degree $i$
in $H^{\ast}(BT; \Lambda) \cong \Lambda[x_{1},\ldots,x_{r}]$ for every $v \in V$ such that
$f(v) \neq 0$.
\end{defn}

\begin{rem}
If the abstract GKM graph $(\Gamma,\alpha)$ comes from a GKM space $(M,J,T)$, then,
by the Chang-Skjelbred Lemma \cite{MR375357}, $H_T^*(M;\QQ)\cong \mathcal{H}_{T}^{\ast}(\Gamma, \alpha; \QQ)$. 
Moreover, under certain restrictions on the disconnected isotropy groups, one has
$\mathcal{H}_T^*(\Gamma,\alpha;\ZZ) \cong H_T^*(M;\ZZ)$ (see \cite{MR2308029}). Without these
conditions the right hand side is in general only isomorphic to a subalgebra of the left
hand side (see \cite{zoller2024integralchangskjelbredcomputationsdisconnected} for concrete formulas). 
\end{rem}

The following Lemma justifies the name ``line Chern numbers" for the integers 
$c_e(f)$ and gives a geometric explanation of Remark \ref{line chern triples}.

\begin{lemma}\label{geom mean line numbers}
With the notation of Section \ref{subsection GKM theory}, let $c_1(\mathbb{L}_i)$ be the first equivariant Chern class of the equivariant line $\mathbb{L}_i$ over $S^2_e$, corresponding to the edge $f_i$, for every $i=1,\ldots,n$. 
Then 
\begin{equation}\label{equality meaning line number}
c_e(f_i)=c_1(\mathbb{L}_i)[S^2_e]\quad \text{for every }i=1,\ldots,n\,.
\end{equation}
\end{lemma}
\begin{proof}
The proof is simple and it involves the Atiyah-Bott-Berline-Vergne localization formula
(in short ABBV formula), see \cite{MR721448, MR685019}. 
We first observe that 
$c_1(\mathbb{L}_i)|_p=\alpha(f_i)$ and $c_1(\mathbb{L}_i)|_q=\alpha(\nabla_e(f_i))$. From the ABBV formula it follows that
$$
\alpha(e)\cdot c_1(\mathbb{L}_i)[S^2_e]=c_1(\mathbb{L}_i)|_p-c_1(\mathbb{L}_i)|_q= \alpha(f_i)-\alpha(\nabla_e(f_i))\,,
$$
and the claim follows from the definition of $c_e(f_i)$ in \eqref{def ce(f)}.
\end{proof}

We observe that, if the GKM graph is $n$-valent, for every edge $e$ there are exactly $n$ line Chern numbers associated to it, each of them
being an integer (see Lemma \ref{geom mean line numbers}). Their sum is another important invariant of the GKM graph. More precisely,
we define the \emph{magnitude} of $e\in E$ as
\begin{equation}\label{def magnitude}
m(e)=\sum_{f\in E_{i(e)}} c_e(f)\,.
\end{equation}
\begin{rem}
Since a connection $\nabla$ gives a bijection between $E_{i(e)}$ and $E_{t(e)}$, from \eqref{eq cefs} it is immediate to see that 
\begin{equation}\label{sym ms}
m(e)=m(\overline{e})\,.
\end{equation}
Furthermore one has
\begin{equation}\label{eq magn}
	m(e)=\frac{\sum_{f\in E_{i(e)}}\alpha(f)-\sum_{\widetilde{f}\in E_{t(e)}}\alpha(\widetilde{f})}{\alpha(e)}
\end{equation}
and we infer that, unlike the line Chern numbers, the magnitude does not depend on the choice of a connection.
\end{rem}
If the abstract GKM graph comes from a GKM space $(M,J,T)$, then the magnitude has the following important geometric meaning.
\begin{lemma}\label{magnitude and c1}
Let $(M,J,T)$ be a GKM space with associated GKM graph $(\Gamma,\alpha)$. Let $c_1$ be the first Chern class of $TM$ and $c_1^T$ its equivariant 
first Chern class. Then 
$$
m(e)=c_1[S^2_e]\quad\text{and}\quad c^T_1(p)-c^T_1(q)=m(e)\alpha(e)
$$  
where $p=i(e)$, $q=t(e)$.
\end{lemma}
\begin{proof}
Observe first that $c_1^T(p)=\displaystyle\sum_{f\in E_p}\alpha(f)$. Then the claims
follow easily from Equation \eqref{eq magn}, the fact that over $S^2_e$ the tangent bundle $TM$ splits equivariantly into the line bundles $\mathbb{L}_i$, and Lemma \ref{geom mean line numbers}\,.
\end{proof}

\subsection{GKM$_k$ spaces}\label{GKMk}
Condition (d) in Definition \ref{D: GKM manifold} is often stated equivalently in the following terms.
Let $M$ and $T$ be as in Definition \ref{D: GKM manifold}, and assume that conditions (a), (b) and (c) hold. Then (d) is equivalent to the following combinatorial condition:
\begin{itemize}
\item[(d')] For each $p\in M^T$, the weights $\alpha_1,\ldots, \alpha_n$ of the isotropy representation of $T$ at $p$ are \emph{pairwise linearly independent}.
\end{itemize}
\begin{rem}\label{equivalence d and d'}
To see that (d) is equivalent to (d') it is enough to apply \cite[Theorem 11.8.1]{MR1689252}, observing that condition (b) implies that $H_T^*(M;\RR)$ is a free
$H^*(BT;\RR)$-module (combine \cite[Theorem 6.5.3]{MR1689252} with \cite[Theorem 6.6.2]{MR1689252} and observe that the universal coefficient theorem implies that from (b) one has $H^{\text{odd}}(M;\RR)=0$).
\end{rem}
GKM spaces, as originally introduced by \cite{MR1489894}, have a natural generalization.
\begin{defn}\label{gkm k def}
Let $M$ and $T$ be as in Definition \ref{D: GKM manifold}, and assume that conditions
(a) and (b) hold. 
Let $k\in \NN$. 
We say that $(M,J,T)$ is a \textbf{GKM$_k$ space} if
\begin{itemize}
\item [($\text{d}_k$)] For all $h\leq \min\{\dim(T),k-1\}$, the \emph{$h$-skeleton} $$\mathcal{S}_{h} := \{ p \in M \colon \dim(T \cdot p)
		\leq h\}$$ is a finite union of $T$-invariant submanifolds of dimension $2h$ and principal orbits dimension equal to $h$.
\end{itemize}
\end{defn}

\begin{rem}
We note that the GKM$_k$ condition implies the GKM$_l$ condition for all $l\leq k$ and that the GKM$_1$ condition is equivalent to having discrete fixed points.
\end{rem}

The next lemma gives a combinatorial analogue of (d$_k$), and is the first definition of GKM$_k$ spaces, as it was originally given by Guillemin and Zara in \cite[Section 2.10]{MR1701922} (see also \cite{MR3456711}).
\begin{lemma}\label{dk equivalent d'k}
Let $M$ and $T$ be as in Definition \ref{D: GKM manifold}, and assume that conditions
(a) and (b) in Definition \ref{D: GKM manifold} hold. 
Let $k\in \NN$. 
Then condition (d$_k$) is equivalent to the following
\begin{itemize}
\item [(d$_k$')] For each $p\in M^T$, any subset of $h$ distinct isotropy weights $\alpha_1,\ldots, \alpha_h$ at $p$ are \emph{linearly independent}, for each $h\leq k$.
\end{itemize}
\end{lemma}
\begin{proof}
 (d$_k$) $\implies$ ($\text{d}'_{k}$) 

Suppose that $\dim(T)=r$, that there is a fixed point $p\in M^T$, and a subset of $h\leq k$ isotropy weights $\alpha_1,\ldots,\alpha_h$ at $p$ that are linearly dependent. In particular $h\leq \dim T$. This implies that 
$$
\mathfrak{k}:=\dim(\ker(\alpha_1)\cap \cdots \cap \ker(\alpha_h))> r-h\,.
$$
Let $K$ be $\exp(\mathfrak{k})\subset T$, then $\dim(K)>r-h$ and $K$ acts trivially on a subspace $V_K \subset T_pM$ of dimension $2h$.
The connected component $N$ of $M^K$ (the submanifold fixed by $K$) which contains $p$ is a submanifold stabilized by $K$. Furthermore, the tubular neighbourhood theorem implies $V_K\subset T_pN$. In particular, $\dim N\geq 2h$.
The principal isotropy type of $N$ contains $K$ so the dimension $\widetilde{h}$ of its orbits satisfies $\widetilde{h}\leq \dim(T)-\dim(K)<h\leq k$. Therefore, for $\widetilde{h}\leq k-1$ we found a submanifold of $\mathcal{S}_{\widetilde{h}}$ of dimension $>2\widetilde{h}$, which contradicts condition (d$_k$).

\vspace{0.3cm}
 ($\text{d}'_{k}$) $\implies$ (d$_k$) 
 
Let $h\leq \min\{r,k-1\}$ with $r=\dim(T)$. Let $x\in \mathcal{S}_h$ be any point and denote by $K\subset T$ its isotropy group. By definition of $\mathcal{S}_h$, we have $\dim K\geq r-h$. Let $N$ be the connected component of $M^K$ which contains $x$. We note that $K$ is the principal isotropy group on $N$.
 
Since $N$ is a connected component of $M^K$, \cite[Threorem 11.6.1]{MR1689252} and the fact that $H_T^*(M;\RR)$ is a free $H^*(BT;\RR)$-module (see Remark \ref{equivalence d and d'}), imply that there exists
a fixed point $p\in M^T\cap N$. Let $\alpha_1,\ldots,\alpha_l$ denote the weights belonging to $T_p N$. We claim that $l\leq h$ or equivalently $\dim(N)\leq 2h$.
Indeed, assume that $l>h$. Then by $(\text{d}'_k)$ we have 
$\dim\{\text{span}\langle \alpha_1,\ldots, \alpha_l \rangle\} \geq h+1$.
In this case
\[\dim\left(\displaystyle\bigcap_{i=1}^l\ker(\alpha_i)\right)\leq r-h-1.\]
But as $K$ acts trivially on $N$, the $\alpha_i$ vanish on $\mathrm{Lie}(K)$, which contradicts the fact that $\dim K\geq r-h$. Consequently, $l\leq h$ holds.

As the $T$-action on $M$ is effective and has a fixed point we also obtain $\dim M\geq
2r\geq 2h$. Therefore we may complete the $\alpha_1,\ldots,\alpha_l$
to a linearly independent set of weights $\alpha_1,\ldots,\alpha_h$ at $p$. Let $K'$ be
the common kernel of the associated characters. The linear independence of the $\alpha_i$
implies, $\dim K'=r-h$. Furthermore $K'\subset K$ since the latter is the common kernel of
the characters belonging to $\alpha_1,\ldots,\alpha_l$. Let $N'$ be the connected
component of $M^{K'}$ which contains $p$. By the tubular neighbourhood theorem and
principal orbit theorem we infer that $K'$ is the principal isotropy type on $N'$. In
particular $N'\subset \mathcal{S}_h$. By the same argument as for $N$ above, we conclude
that $\dim(N')\leq 2h$. But $T_p N'$ contains the irreducible representations belonging to
$\alpha_1,\ldots,\alpha_h$ so in fact $\dim(N')=2h$ and $x\in N\subset N'$. As $M$ is
compact, the number of $N'$ arising in this way is finite.

\end{proof}

\begin{rem}\label{rem: compl0}
In condition (d$_k$), the $T$-action on each of the $2h$-dimensional isotropy submanifolds $N\subset \mathcal{S}_h$ has a kernel $K\subset T$ of dimension $\dim (T)-h$, as the principal orbit dimension of $N$ is equal to $h$. In particular the quotient torus $T/K$ is of dimension $h$ and acts effectively on $N$.
\end{rem}

\begin{defn}\label{def Ham GKMk}
Let $(M,\omega,\mu)$ be a compact Hamiltonian $T$-space. We say that 
$(M,\omega,\mu)$ is a \textbf{compact Hamiltonian GKM$_k$ space}, for some
$k\in \NN$, if condition (d$_k$) above holds. 
\end{defn}

\begin{rem}\label{gkmk ham}
In the case of a Hamiltonian action, the isotropy submanifolds are again Hamiltonian $T$-manifolds. In view of Remark \ref{rem: compl0}, condition (d$_k$) is then equivalent to saying that $\mathcal{S}_h$ is a finite union of $2h$-dimensional toric manifolds (when considered with the effective quotient action induced by the $T$-action), for each $h\leq \min\{\dim(T),k-1\}$.
\end{rem}

\begin{rem}\label{rem: gkm3connection}
As connections play a large role in this paper, it is important to note that under the GKM$_k$-condition for $k\geq 3$, the connection is always unique. Indeed, if, for some adjacent $e,f\in E(\Gamma)$, there were two distinct possible choices $h_1,h_2$ for the image $\nabla_e(f)$, then it would follow that
\[\alpha(h_1) \equiv \alpha(f)\equiv \alpha(h_2)\mod \alpha(e)\]
which would violate the GKM$_3$ condition as it implies that $\alpha(h_1),\alpha(h_2),\alpha(e)$ are linearly dependent.
\end{rem}

\subsection{Compact positive monotone Hamiltonian GKM$_k$ spaces}
Let $(M,\omega)$ be a compact symplectic manifold. Since the set of almost complex structures compatible with
$\omega$ is contractible, it is possible to consider the Chern classes $c_{2j}\in H^{2j}(M;\ZZ)$ of $(TM,J)$ as invariants of $(M,\omega)$. 
As defined in the introduction, we say that $(M,\omega)$ is \emph{positive monotone} if $c_1=[\omega]$.
Since we deal with symplectic manifolds with Hamiltonian group actions, we want to see
what the implications of this condition
in Hamiltonian geometry are. 

Suppose that a compact torus of rank $r$ is acting on $(M,\omega)$ in a Hamiltonian way. This means that there exists a moment
map $\mu\colon M \to \mathfrak{t}^*$ satisfying
\begin{equation}\label{def moment map 2}
d\langle \mu, \xi \rangle = - \iota_{\xi^\#} \omega \quad \text{for all  }\xi \in \mathfrak{t},
\end{equation}
where $\xi^\#$ is the vector field associated to $\xi$ and $\langle \cdot , \cdot \rangle$ is the natural pairing between $\mathfrak{t}^*$ and $\mathfrak{t}$. 
We recall that, unless otherwise stated, we assume that the torus action is effective and call the triple $(M,\omega,\mu)$ a \textit{Hamiltonian $T$-space}. 
Moreover in this paper we are only concerned with the case in which the fixed point set $M^T$ is finite. 

The positive monotonicity has a very natural translation in equivariant terms. Indeed, the existence of a Hamiltonian action implies that both $[\omega]$ and $c_1$ admit equivariant extensions
 (as equivariant differential forms in the Cartan model for
 equivariant cohomology),
namely classes $[\omega-\mu]$ and $c_1^T$ in $H^2_T(M;\RR)$ that restrict to $[\omega]$ and $c_1$ respectively, where the restriction is given by the natural
map $\rho\colon H^*_T(M;\RR)\to H^*(M;\RR)$ induced by the inclusion of the trivial group $\{e\}$ in $T$. The condition $c_1=[\omega]$ implies that, modulo
an element of $\mathfrak{t}^*$, the equivariant extensions are equal. Since the moment map
can be modified by a constant, it is not restrictive to assume that 
\begin{equation}\label{equivariant monotonicity}
c_1^T=[\omega-\mu]\,.
\end{equation}
We observe that $c_1^T(p)=\sum_{i=1}^n \alpha_i(p)$ for every fixed point $p$, where the $\alpha_i(p)$s are the weights of the isotropy representation of $T$ at $p$, and it can be easily proved
that \eqref{equivariant monotonicity} is equivalent to 
\begin{equation}\label{weight sum formula}
c_1^T(p)=\sum_{i=1}^n \alpha_i(p) = -\mu(p)\, \quad \text{for all }p\in M^T\,.
\end{equation}
Formula \eqref{weight sum formula} is referred to as the \emph{weight sum formula} (see \cite[Section 3.1]{MR5060441}).
\begin{defn}
A \textbf{compact positive monotone Hamiltonian GKM$_k$ space} is a Hamiltonian $T$-space $(M,\omega,\mu)$ such that the $T$ action is GKM$_k$ and such that 
$c_1=[\omega]$. 

Without loss of generality we also assume that the weight sum formula \eqref{weight sum formula} holds.
\end{defn}
Suppose that the moment map is injective if restricted to the fixed point set. 
We note that, since the weight sum formula implies the positive monotonicity, the latter can be checked directly from the GKM graph embedded in $\mathfrak{t}^*$ via the moment map: it is sufficient to
check that \eqref{weight sum formula} holds at every vertex $p\in V \cong M^T$; this kind of graphs are also called \emph{reflexive GKM graphs} (see \cite[Section 5.3]{MR3695881} as well as the figures therein).

\section{A bound on the line Chern numbers}
\label{sec: CLN bound}

A central step in obtaining our main results is to establish control over the line Chern numbers of a compact positive monotone Hamiltonian GKM space. 
In what follows, we recall several fundamental facts required for these results.

Suppose that $(M,\omega,\mu)$ is a compact Hamiltonian GKM space and $(\Gamma, \alpha)$ its
GKM graph. Then the set $\mathcal{S}=\{S^2_e\}_{e\in E}$ of $T$-invariant symplectic spheres, in
one-to-one correspondence with the undirected edges of $\Gamma$, is a toric one-skeleton in the sense
of \cite[Def. 4.10, Lemma 4.11]{MR3695881} (see also \cite[Def. 4.13]{MR4940256}). Therefore, combining 
\cite[Thm. 1.5]{MR3695881} with Lemma \ref{magnitude and c1} one obtains immediately the following
\begin{prop}\label{magnitudes topological}
Let $(M, \omega, \mu)$ be a compact Hamiltonian GKM space with GKM graph $(\Gamma,\alpha)$. Let $2n$ be the dimension of $M$ and $\mathbf{b}=(b_0,b_2,\ldots,b_{2n})$ 
be the vector of its even Betti numbers. Then the sum of the magnitudes only depends on $\mathbf{b}$, more precisely
\begin{equation}\label{sum magn}
\sum_{e\in E}m(e)= \sum_{j=0}^n b_{2j}\Big[ 6j(j-1)+\frac{5n-3n^2}{2}\Big]
\end{equation}
\end{prop}
\begin{rem}\label{remark magnitudes}
The right hand side of \eqref{sum magn} is nothing else but the constant $C(n,\mathbf{b})$
in \cite[Thm. 1.5]{MR3695881} (see its proof). To be coherent with this, we
also define in this paper
\begin{equation}\label{definition Cnb}
C(n,\mathbf{b}):=\sum_{j=0}^n b_{2j}\Big[ 6j(j-1)+\frac{5n-3n^2}{2}\Big]\,.
\end{equation}
\end{rem}
If $(M,\omega,\mu)$ is a compact positive monotone Hamiltonian GKM space with $c_1=[\omega]$, then one has 
\begin{equation}\label{magn pos}
m(e)=c_1[S^2_2]=\int_{S^2_e} \omega >0\,,
\end{equation}
as $S^2_e$ is a symplectic sphere. Therefore
\eqref{sum magn} has the following consequences:
\begin{enumerate}
\item[(i)] For each fixed $\mathbf{b}=(b_0,b_2,\ldots,b_{2n})$ there are only \emph{finitely many possible magnitudes}, as they are positive integers partitioning the right hand side of \eqref{sum magn}. In particular 
\begin{equation}\label{upper bound ms}
0< m(e)< C(n,\mathbf{b}) \quad \text{for all }e\in E\,.
\end{equation}
The upper bound in the above equation could be refined using the \emph{index} $k_0$ of $(M,\omega)$, i.e.\ 
the largest positive integer dividing $c_1$ in $H^2(M;\ZZ)$ (Note that $c_1=[\omega]$ implies that $c_1\neq 0$). Indeed, $k_0$ divides all magnitudes $m(e)$ so from $\sum m(e)= C(n,\mathbf{b})$ and $k_0\leq m(e)$ we get
\[k_0\leq m(e) \leq C(n,\mathbf{b}) - (|E|-1)k_0\]
with $|E|$ denoting the number of edges. As this has no major effect on the growth rates of our quantitative results, we do not pursue this refinement and settle to work with Equation \eqref{upper bound ms} in what follows.

\item[(ii)] We note that, since the left hand of \eqref{sum magn} is positive, the
	right hand side must be positive too. This, together with other arguments, give
	inequalities on the Betti numbers of a compact positive monotone Hamiltonian GKM space of a certain
	dimension. (For a finer analysis of these inequalities, see for instance 
	\cite[Section 5]{MR3695881} and \cite[Section 4.2]{MR4940256}.)
\end{enumerate}

In this paper we focus on a refinement of (i) for GKM$_3$-spaces.

\begin{thm}\label{thm: mainthm}
Let $n\in \NN$, $n\geq 3$, and $\mathbf{b}\in \NN^{n+1}$. Then for each $n$ and $\mathbf{b}$, there are only finitely many possible values for the line Chern numbers of a compact 
positive monotone
Hamiltonian GKM$_3$ space $(M,\omega,\mu)$ of dimension $2n$ and with even Betti numbers $(b_0,b_2,\ldots,b_{2n})=\mathbf{b}$. 

More precisely, let $C(n,\mathbf{b})$ be the constant defined in \eqref{definition Cnb}. 
Then the following bounds hold for any line Chern number $c_e(f)$ of pairs of distinct adjacent edges $e,f$  in the GKM graph: 
\begin{align}
\label{bound chern line numbers} (3-n)C(n,\mathbf{b})-2\leq c_e(f)< C(n,\mathbf{b})
\end{align}
\end{thm}

We note that the case $n\leq 2$, which is excluded above, can be settled separately through the following

\begin{rem}\label{rem open close toric}
By Remark 
\ref{line chern triples} \eqref{alpha e e}, $c_e(e)$ is always 2 and for $n=1$ the only 
GKM space is the sphere with standard toric action, whose GKM graph has one edge $e$ and the only line Chern number is $c_e(e)=2$.
For $n=2$ a compact positive monotone Hamiltonian GKM$_3$ space is just a symplectic toric manifold, therefore 
a smooth toric Fano 2-fold, and the corresponding moment map is a smooth reflexive polygon. Their classification \cite{Ewald1988} asserts that there are, modulo a suitable notion of isomorphism, only 5 such manifolds and their corresponding line Chern numbers satisfy  
$$
-1\leq c_e(f)\leq 1
$$ 
for any pair of distinct adjacent edges $e,f$ in the associated GKM graph. 

For $n\geq 3$ the classification of smooth toric Fano $n$-folds becomes more complicated. Since any such variety
is a positive monotone Hamiltonian GKM$_n$ space, we believe that \eqref{bound chern line numbers} and \eqref{eq:global bound} can have interesting applications in their classification (see also Section \ref{sec application reflexive}).
\end{rem}

\begin{proof}[Proof of Theorem \ref{thm: mainthm}]
We consider  the unique connection $\nabla$ (cf. Remark \ref{rem: gkm3connection}) on the
GKM graph $(\Gamma,\alpha)$ of $M$, with $\Gamma=(V,E)$. Let $v\in V$ be the vertex
corresponding 
to the fixed point $p\in M^T$, and
$e,f$ be edges in $E_v$ and $h=\nabla_e(f)$. By definition of $c_e(f)$ we have 
\begin{equation}\label{alpha h}
\alpha(h)=\alpha(f)-c_e(f)\alpha(e)
\end{equation}
 (see \eqref{def ce(f)}) and we intend to bound $c_e(f)$. 
 For each $\beta\in \ZZ_{\mathfrak{t}}^*$ define
 $$
 \mathfrak{h}_{\beta}:=\{\xi\in \mathfrak{t}\mid \beta(\xi)\in \ZZ\}\,.
 $$
Let $H\subset T$ be the codimension $2$ subgroup defined as $H:=\exp( \mathfrak{h}_{\alpha(e)}\cap
\mathfrak{h}_{\alpha(f)})$. Since $(M,\omega,\mu)$ is a GKM$_3$ Hamiltonian space (see Definition \ref{gkm k def}), 
the connected component $K$ of $M^H$ containing $p\in M^T$ is a $4$-dimensional symplectic submanifold with an effective Hamiltonian action of the 2-dimensional
quotient torus $T/H$, whose moment map $\mu_K$ can be identified with 
$\mu|_{K}\colon K\to \mathfrak{t}^*$ (see for instance the discussion after Remark 2.23 in \cite{MR5060441}). 
Therefore $(K,\omega|_K,\mu_K)$ is a symplectic toric submanifold, the action is GKM, and its
GKM graph $\Gamma_K$ can be identified with the 1-skeleton of the moment map image $\mu(K)=:P$. 

We observe that the moment map $\mu_K$ is injective on the fixed point set of $K$, and we identify 
the latter with the set of vertices of $\Gamma_K$ using the moment map $\mu_K=\mu|_K$. 
Consider now $p,q,r,s\in M^T$ such that $q=t(f)$, $r=t(e)=i(h)$ and $s=t(h)$ (see Figure \ref{cone mm}).
By Lemma \ref{magnitude and c1} and the weight sum formula \eqref{weight sum formula} 
\begin{align}\label{eq:edge vector}
\mu(q)-\mu(p)=\; & c_1^T(p)-c_1^T(q)=m(f)\alpha(f),
\end{align}
and analogously, using that $h=\nabla_e(f)$ and \eqref{alpha h},
\begin{align}
\label{eq2}
\mu(r)-\mu(p)=\;& m(e)\alpha(e),\\
\label{eq3}
\mu(s)-\mu(r)=\;& m(h)\alpha(h)=m(\nabla_e(f))\big(\alpha(f)-c_e(f)\alpha(e)\big)\,.
\end{align}
\begin{figure}[h]
	\centering
	\includegraphics{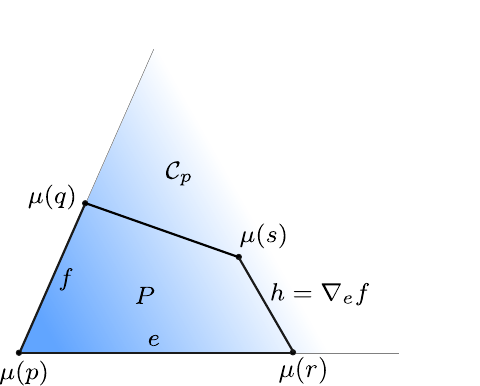}
	\caption{The cone $\mathcal{C}_{p}$ and the moment image $P$ of $K$.} 
	\label{cone mm}
\end{figure}
Let $\mathcal{C}_p$ be the two-dimensional cone in $\mathfrak{t}^*$ given by 
$\mathcal{C}_p:=\mu(p)+\{a\,\alpha(e)+b\,\alpha(f), a,b\in \RR_{\geq 0}\}$. By convexity of
$P$ we have that $P\subset \mathcal{C}_p$ and in particular
$\mu(s)\in \mathcal{C}_p$, which means that $\mu(s)-\mu(p)$ must be a non-negative combination of 
$\alpha(e)$ and $\alpha(f)$. Using \eqref{eq2} and \eqref{eq3} we have that
$$
\mu(s)-\mu(p)=m(\nabla_e(f))\big(\alpha(f)-c_e(f)\alpha(e)\big)+m(e)\alpha(e)
$$
which, by what we observed before, implies that 
\begin{equation}\label{first bound m}
m(e)-m(\nabla_e(f))c_e(f)\geq 0\,.
\end{equation}
We obtain 
\begin{equation}\label{first upper bound cef}
c_e(f)\leq m(e)< C(n,\mathbf{b})\,,
\end{equation} 
where the first inequality follows from $m(\nabla_e(f))\geq 1 $ and \eqref{first bound m},
in case $c_e(f)\geq 0$, and from $m(e)\geq 1$ in case $c_e(f)<0$, whereas the
latter inequality is exactly \eqref{upper bound ms}. 
Having achieved this upper bound for all $e\neq f \in E_p$ we also get a lower bound. Indeed, using \eqref{alpha e e}, \eqref{def magnitude}, and \eqref{first upper bound cef}, we have
\begin{equation}
\label{eq: lower bound CLN}
\begin{aligned}
	c_e(f) &= m(e)-c_e(e)-\sum_{f'\in E_p\backslash\{e,f\}} c_e(f')\\
				 &= m(e)-2-\sum_{f'\in E_p\backslash\{e,f\}} c_e(f')\\
				 &\geq -2 -(n-3)m(e),
\end{aligned}
\end{equation}
where the latter is greater or equal to $-2-(n-3)C(n,\mathbf{b})$ when $n\geq 3$. 
\end{proof}

\begin{rem}
From \eqref{eq cefs}, \eqref{sym ms} and \eqref{first bound m}, it easily follows that 
$$
c_e(f)\leq \min \left\{\frac{m(e)}{m(f)},\frac{m(e)}{m(\nabla_e(f))}\right\}\,.
$$
\end{rem}
The line Chern numbers can also be bounded in terms of the Euler characteristic as follows

\begin{cor}\label{global bound}
Let $n,\chi \in \NN$ with $n\geq 3$. Then there are only finitely many possible values for the line Chern numbers of a compact positive monotone
Hamiltonian GKM$_3$ space $(M, \omega, \mu)$ of dimension $2n$ and Euler characteristic $\chi$. 

More precisely one has that the following bounds hold for any pair of distinct adjacent edges $e,f$ in the GKM graph:
\begin{align}
\label{eq:global bound} (3-n)\chi\cdot \frac{n(n+1)^2}{2}-2\leq c_e(f)< \chi\cdot\frac{n(n+1)^2}{2}
\end{align}

\end{cor}
\begin{proof}
It is easy to see that, since $b_{2j}\leq \chi$ for all $j=0,\ldots,n$, for every vector of even Betti numbers $\mathbf{b}$ one has
\begin{equation}\label{bound C and chi}
C(n,\mathbf{b})=\sum_{j=0}^n b_{2j}\Big[ 6j(j-1)+\frac{5n-3n^2}{2}\Big]\leq \chi \sum_{j=0}^n \Big[ 6j(j-1)+\frac{5n-3n^2}{2}\Big] = \chi\cdot\frac{n(n+1)^2}{2}
\end{equation}
and the inequality \eqref{eq:global bound} is a direct consequence of those in Theorem \ref{thm: mainthm}.
\end{proof}

In the next Proposition we prove that the bounds on the line Chern numbers are independent on $\chi$
if $(M,\omega)$ satisfies some additional hypotheses.

\begin{prop}\label{P:removing-chi}
Let $(M,\omega,\mu)$ be a compact connected positive monotone Hamiltonian
$\mathrm{GKM}_3$ space of dimension $2n\geq 6$, with index
$k_0\geq n-3$. Assume that its even Betti numbers are unimodal, that is,
\[
1=b_0\leq b_2\leq\cdots\leq b_{2\lfloor n/2\rfloor}.
\]
Then the line Chern numbers are bounded in terms of $n$ only. More precisely,
for any pair of distinct adjacent edges $e,f$ in the GKM graph,
\[
-2-(n-3)(2n^2+3n-3)\leq c_e(f)\leq 2n^2+3n-3.
\]
\end{prop}

\begin{proof}
By \cite[Corollary~1.1]{MR3665359}, applied to a generic component of the
moment map, the index satisfies $k_0\leq n+1$. Moreover an Hamiltonian GKM action
always admits a toric $1$-skeleton (see \cite[Section~4.1]{MR3695881}).
Therefore we can use the arguments in 
\cite[Corollary~5.5]{MR3695881} and refer the reader to it.

Writing $m(e)$ for the magnitude of an edge, we have
\[
\sum_{e\in E}\bigl(m(e)-k_0\bigr)=C(k_0,n,\mathbf{b}).
\]
In particular, the estimate obtained from
\cite[Equation~(5.6)]{MR3695881} is
\[
0\leq C(k_0,n,\mathbf{b})\leq A_0+S\,b_{2L},
\]
where
\[
A_0=n(3n-k_0-1),\quad\quad
S=\frac{n(n-1)}2(n-k_0-3),
\]
and $b_{2L}$ is the even Betti number appearing in \cite[Corollary~5.5]{MR3695881}.

Since $k_0\geq n-3$, we have $S\leq0$; moreover $b_{2L}\geq1$.
Consequently,
\begin{align*}
0\leq \sum_{e\in E}\bigl(m(e)-k_0\bigr)
&\leq A_0+S\\
&=n(3n-k_0-1)+\frac{n(n-1)}2(n-k_0-3)\\
&=\frac{n(n+1)(n+1-k_0)}2.
\end{align*}
As recalled in the proof of \cite[Corollary~5.5]{MR3695881}, the index
$k_0$ divides every positive magnitude $m(e)$. Thus $m(e)\geq k_0$
for every edge, so each summand above is nonnegative. It follows that
\begin{equation}\label{bound m(e) prop}
m(e)\leq k_0+\frac{n(n+1)(n+1-k_0)}2=:B_n(k_0).
\end{equation}
We observe that the function
\[
B_n(k)=\frac{n(n+1)^2}{2}+\left(1-\frac{n(n+1)}2\right)k
\]
is non-increasing in $k$ for $n\geq3$. Using $k_0\geq n-3$, we obtain
\begin{equation}\label{Bnk}
B_n(k_0)\leq B_n(n-3)
=2n^2+3n-3.
\end{equation}
Finally, Equations~\eqref{first upper bound cef} and
\eqref{eq: lower bound CLN}, together with \eqref{bound m(e) prop}, give respectively
\[
c_e(f)\leq B_n(k_0)\quad\text{and}\quad
c_e(f)\geq -2-(n-3)B_n(k_0).
\]
Using \eqref{Bnk} we obtain the claim.
\end{proof}

\begin{cor}\label{C: dim 8 bounded cln}
Let $(M, \omega, \mu)$ be a compact positive monotone Hamiltonian GKM$_3$ space of
dimension $8$. Then the Chern line numbers are bounded by
\[
	-43 \leq c_{e}(f) \leq 41.
\]
\end{cor}
\begin{proof}
By \cite[Theorem 1.3]{MR3275642} the Betti numbers of $(M, \omega, \mu)$ are unimodal, and the above
bounds follow directly from Proposition \ref{P:removing-chi}.
\end{proof}

Line Chern numbers do not contain the full information on the labels of a GKM graph, therefore it is useful to introduce the following weaker notion of equivalence, the implications of which will be studied further in Section \ref{sec: max extensions}.

\begin{defn}
We call two GKM graphs $(\Gamma,\alpha)$ and $(\Gamma',\alpha')$ with fixed connections
$\nabla$ and $\nabla'$ respectively
\emph{virtually isomorphic} if there is an isomorphism of graphs
$$\Phi\colon \Gamma=(V,E)\rightarrow \Gamma'=(V',E')$$  such that:
\begin{itemize}
\item[(i)] $\Phi$ is \emph{compatible with the connection}, namely 
$$
\Phi(\nabla_e(f))=\nabla'_{\Phi(e)}\Phi(f)\quad\text{for all }e,f\in E_p,\;\;\ \text{for all }p\in V
$$
\item[(ii)] $\Phi$ \emph{preserves the line Chern numbers}, namely
$$
c_{e}(f)=c_{\Phi(e)}(\Phi(f)) \quad\text{for all }e,f\in E_p,\;\;\ \text{for all }p\in V\,.
$$
 \end{itemize}
\end{defn}

Given the definition above, the following is a direct consequence of Corollary \ref{global bound}.
\begin{cor}\label{cor: main}
For each $n, \chi\in \NN$,
there are only finitely many virtual isomorphism classes of GKM graphs of compact positive monotone Hamiltonian GKM$_3$ spaces  of Euler characteristic $\chi$ and dimension $2n$.
\end{cor}

\begin{proof}
Let $(M,\omega,\mu)$ be a $2n$-dimensional compact positive monotone Hamiltonian GKM$_3$ space with Euler characteristic $\chi$. 

Since $\chi$ is exactly the number of vertices of the associated GKM graph $(\Gamma,\alpha)$ and $\Gamma$ is $n$-regular, there are only finitely many admissible 
combinatorial candidates for the unlabeled graph $\Gamma$ and hence also for $\nabla$ (up to graph isomorphism preserving the connection). 
Once $(\Gamma,\nabla)$ is fixed, Corollary \ref{global bound} asserts that there are only finitely many possibilities for the line Chern numbers, which concludes the proof. 
\end{proof}

\begin{ex}\label{line chern numbers not Ham} The arguments in this section rely on the upper bound on the magnitudes, which is achieved
using the positive monotone GKM condition (see Remark \ref{remark magnitudes}), and the upper bound on the line Chern numbers, which is achieved using 
the convexity properties of Hamiltonian GKM$_3$ actions (see Theorem \ref{thm: mainthm}). 
In this example we provide a family of abstract GKM$_3$ graphs such that the associated magnitudes are positive and constant within the family, while the line Chern numbers are unbounded. 
In particular, this shows that this family of abstract GKM$_3$ graphs does not come from Hamiltonian actions, and that the Hamiltonian condition in Theorem \ref{thm: mainthm} is necessary. 

Fix a basis $\alpha,\beta,\gamma$ of $\ZZ^3$, let $k\in \ZZ$ and consider the following abstract GKM$_3$ graph, oriented as in the picture.

\begin{center}
		\begin{tikzpicture}
		
		\usetikzlibrary{decorations.markings}
		
		\begin{scope}[very thick, decoration={markings, mark=at position 0.5 with {\arrow{>}}}]

			\node at (0,0)[circle,fill,inner sep=2pt] {};
			\node at (1,1)[circle,fill,inner sep=2pt]{};
			\node at (1,-1)[circle,fill,inner sep=2pt]{};
			\node at (3,1)[circle,fill,inner sep=2pt]{};
			\node at (3,-1)[circle,fill,inner sep=2pt]{};
			\node at (4,0)[circle,fill,inner sep=2pt]{};
			\draw[postaction={decorate}] (0,0)--++(1,1);
			\draw[postaction={decorate}] (3,1)--++(-2,0);
			\draw[postaction={decorate}] (4,0)--++(-1,1);
			\draw[postaction={decorate}] (4,0)--++(-1,-1);
			\draw[postaction={decorate}] (3,-1)--++(-2,0);
			\draw[postaction={decorate}] (0,0)--++(1,-1);
			\draw[postaction={decorate}] (1,1) -- ++(0,-2);
			\draw[postaction={decorate}] (3,1) -- ++(0,-2);
			
			\draw[postaction={decorate}] (4,0) --++ (2,0);

			\node at (6,0)[circle,fill,inner sep=2pt] {};
			\node at (7,1)[circle,fill,inner sep=2pt]{};
			\node at (7,-1)[circle,fill,inner sep=2pt]{};
			\node at (9,1)[circle,fill,inner sep=2pt]{};
			\node at (9,-1)[circle,fill,inner sep=2pt]{};
			\node at (10,0)[circle,fill,inner sep=2pt]{};
			\draw[postaction={decorate}] (6,0)--++(1,1);
			\draw[postaction={decorate}] (9,1)--++(-2,0);
			\draw[postaction={decorate}] (10,0)--++(-1,1);
			\draw[postaction={decorate}] (10,0)--++(-1,-1);
			\draw[postaction={decorate}] (9,-1)--++(-2,0);
			\draw[postaction={decorate}] (6,0)--++(1,-1);
			\draw[postaction={decorate}] (7,1) -- ++(0,-2);
			\draw[postaction={decorate}] (9,1) -- ++(0,-2);
			
			\draw[very thick] (0,0) --++ (-0.75,0);
			\draw[very thick, dashed, ->] (-1.5,0) --++ (0.75,0);
			\draw[very thick, ->] (10,0) --++ (0.75,0);
			\draw[very thick, dashed] (10.75,0) --++ (0.75,0);
			
			\node at (5,0.3) {$\alpha$ };
			\node at (8,1.3) {$\alpha$};
			\node at (8,-1.3) {$\alpha$};
			\node at (2,1.3) {$\alpha$};
			\node at (2,-1.3) {$\alpha$};
			\node at (-1,0.3) {$\alpha$};
			\node at (0.3,0.8) {$\beta$};
			\node at (0.3,-0.8) {$\gamma$};
			\node at (2,0) {$\gamma-\beta$};
			\node at (0.3,0.8) {$\beta$};
			\node at (0.3,-0.8) {$\gamma$};
			
			\node at (3.7,0.8) {$\beta$};
			\node at (3.7,-0.8) {$\gamma$};
			
			\node at (6,0.8) {$\beta+k\alpha$};
			\node at (6,-0.8) {$\gamma-k\alpha$};
			\node at (10.1,0.8) {$\beta+k\alpha$};
			\node at (10.1,-0.8) {$\gamma-k\alpha$};
			
			\node at (8,0) {\begin{tabular}{c}
			$\gamma-\beta$ \\ $-2k\alpha$
			\end{tabular}};

		\end{scope}

		\end{tikzpicture}
		
	\end{center}

Here we identify the right most with the left most horizontal edge.
The existence of a (unique) compatible connection can be checked for each edge independently. The magnitudes are $3$ for all edges involved in triangles and $2$ for the horizontal edges. However the line Chern numbers of the central and right most horizontal edges are $\pm k$.

We do not know whether this family of GKM graphs is in fact realized
by 6-dimensional almost complex manifolds. However observe some facts which are consistent
with the geometric side: The sum of all the magnitudes is 48. If this abstract
GKM$_3$ graph came from a $J$-preserving torus action with isolated fixed points
on a compact almost complex manifold $(M,J)$ of dimension 6, then the sum of all the
magnitudes would also be 48. Indeed, from the picture it is easy to see that, with respect
to this chosen orientation, the number $N_0$ of vertices with 0 entering edges is 2. From
\cite[Remark 2.10]{MR776689} we know that $\mathrm{Todd}(M)=N_0$, and in dimension 6 one has
$\mathrm{Todd}(M)=\frac{c_1c_2[M]}{24}$, yielding $c_1c_2[M]=24\,N_0=48$. Finally, from
\cite[Prop. 4.6]{MR3230015} it follows that $c_1c_2[M]$ is exactly the sum of all the magnitudes
of the edges of a multigraph associated to the action which, in this case, can be taken to
be exactly the GKM graph itself.

Finally, it is interesting to note that, after restricting the weights to a generic
hyperplane, the graphs can be realized by simply connected smooth (not necessarily almost
complex) GKM $T^2$-spaces (see \cite{2210.01856v1}). 
\end{ex}

\begin{ex}\label{ex: weird cube}
We give a family of examples in order to prove that Theorem \ref{thm: mainthm} is wrong if one assumes that the action is only GKM$_2$ and not GKM$_3$. Indeed, if we restrict the standard toric action on $M=(S^2)^3$ to a suitable rank $2$-subtorus, one can obtain the following
GKM graph, for an arbitrary $k\in \ZZ$:
\begin{center}
\begin{tikzpicture}
\node at (0,0)[circle,fill,inner sep=2pt] {};
\node at (2,0)[circle,fill,inner sep=2pt] {};
\node at (0,-2)[circle,fill,inner sep=2pt] {};
\node at (2,-2)[circle,fill,inner sep=2pt] {};

\node at (6,-2)[circle,fill,inner sep=2pt] {};
\node at (8,-2)[circle,fill,inner sep=2pt] {};
\node at (6,-4)[circle,fill,inner sep=2pt] {};
\node at (8,-4)[circle,fill,inner sep=2pt] {};

\draw[very thick] (0,0) --++ (2,0)--++(0,-2)--++(-2,0)--++(0,2);
\draw[very thick] (6,-2) --++ (2,0)--++(0,-2)--++(-2,0)--++(0,2);
\draw[very thick] (0,0) --++ (6,-2);
\draw[very thick] (2,0) --++ (6,-2);
\draw[very thick] (2,-2) --++ (6,-2);
\draw[very thick] (0,-2) --++ (6,-2);

\node at (-0.4,-1) {$\alpha$};
\node at (1,0.3) {$\beta$};
\node at (5,-0.5) {$\alpha+k\beta$};

\end{tikzpicture}
\end{center}
Here $\alpha,\beta$ is a basis of $\ZZ^2$ and parallel edges share the same label. The standard connection inherited from the toric example satisfies 
$c_e(f)=0$ for all distinct adjacent $e, f$. However one could also take the following ``non-toric connection'' along the edge $e$ of label $\beta$, where the edge $f$ labelled $\alpha$ is sent to the edge $f'$ labelled $\alpha+k\beta$, which gives $|c_e(f)|=|k|$. 

As shown in Lemma \ref{geom and comb connection}, the modified connection does in fact come from a different splitting of
$TM|_{S^2_e}$ into $T^2$-invariant line bundles. This splitting is however not respected
by the $T^3$-action (for which a choice of equivariant splitting is indeed unique).
\end{ex}

\section{Bounds on the moment map image}
\label{sec: max extensions}

The goal of this section is to use the previously established bounds on the line Chern
numbers to derive a bound on the size of the moment image. As it turns out this is not
possible in general, since one can achieve arbitrarily large moment map images, even up to
integral automorphisms, by restricting to subgroups (see Example \ref{ex: long weights}).
To rule out this phenomenon we will consider actions which are maximal in a certain sense.
To make this notion of maximality precise, one fixes the line Chern numbers and studies
the group of admissible axial functions, as it has been done by Kuroki in \cite{MR3943448}. In Section
\ref{subsec: max extensions} we describe the constructions and develop the terminology in
a way that is suited for our purpose; we stress however that the underlying ideas of
Section \ref{subsec: max extensions} appear in \cite{MR3943448}. In the subsequent Section \ref{subseq: box} we prove a theorem on the ``size of maximal extensions'' (Theorem \ref{thm: Box}) that leads to a bound of the moment map image for certain Hamiltonian actions (Corollary \ref{cor: box}).

\subsection{Construction of maximal extensions}\label{subsec: max extensions}

Consider a GKM graph $(\Gamma,\alpha)$, where $\alpha$ takes values in $\ZZ^k$. After
fixing a compatible connection $\nabla$, we obtain the associated line Chern numbers
$c_e(f)$. Conversely, assume we are given an $n$-valent graph $\Gamma=(V,E)$ together with
a family of bijections $\nabla=\{\nabla_e\colon E_{i(e)} \to E_{t(e)}\mid e\in E\}$
satisfying $\nabla_e(e)=\bar{e}$ and
$\nabla_{\overline{e}}=\nabla_e^{-1}$, and an integral function $c_*\colon (e,f)\mapsto
c_e(f)\in \ZZ$ which is defined for any pair of edges $e,f\in E_p$, for all $p\in V$.
The $c_e(f)$ are candidates for line Chern numbers and we will try to find a ``maximal axial function'' which is compatible
with the given $\nabla$ and has these integers as line Chern numbers.
\begin{rem}
 It makes sense to additionally impose the relations discussed in Remark \ref{line chern triples} for the $c_e(f)$. This does not enter directly in this section, but if the $c_e(f)$ do not satisfy these relations, then the axial function constructed below can never belong to a GKM graph with a fixed connection. 
\end{rem}

We define an axial function $\gamma$ associated to $(\Gamma,\alpha,\nabla)$ and $c_*$, and taking values in a free Abelian group, as follows: Consider $\ZZ^{E}$, the group of $\ZZ$ labels of the edges of $\Gamma$, regarded as the group of maps from $E$ to $\ZZ$. We note that, on the geometric side, an element of $\ZZ^{E}$ 
can be thought of as a family of (independent) weights associated to a circle action on the family of spheres corresponding to the elements of $E$. 

Now one can consider the subgroup $L=L(\Gamma,\nabla, c_*)$ of all such functions $\beta\colon E\rightarrow \ZZ$ satisfying the conditions
\begin{enumerate}
\item $\beta(e)=-\beta(\overline{e})$;
\item for any $p\in V$ and any $e,f\in E_p$ one has $$\beta(f)- \beta(\nabla_e(f))=c_e(f)\beta(e)\,.$$
\end{enumerate}

\begin{rem}\label{rem: system of eq}
We note that, for a fixed enumeration of the edges, one obtains an isomorphism $\ZZ^{E}\cong \ZZ^a$, where $a=|E|$, and $L$ is the kernel of a matrix $A\in \ZZ^{M\times a}$, where every row corresponds to one of the equations above. In particular,
for the rows corresponding to the equations of type (1), all of the entries are 0 except for two of them, which are $1$ and $-1$. For the rows
corresponding to the equations of type (2), all of the entries are 0 except for three of
them, which are $1$, $-1$ and the corresponding line Chern number.
\end{rem}

As a subgroup of a free $\ZZ$-module, $L$ is again a free $\ZZ$-module. Let $\gamma_1,\ldots,\gamma_m$ be a Basis of $L$ and consider the map $\gamma\colon E\rightarrow \ZZ^m$ defined by
\[\gamma(e):=(\gamma_1(e),\ldots,\gamma_m(e))\]
\begin{rem}\label{rem: transpose business}
Reformulated in terms of matrices, the relation between the chosen basis of $L$ and the resulting weights of the labelled graph $(\Gamma,\gamma)$ are as follows: Enumerating the edges via an isomorphism $\ZZ^E\cong \ZZ^a$, $a\in \NN$ and writing the basis of $L$ as the rows of a matrix $B\in \ZZ^{m\times a}$, then the weights will be the \emph{columns} of $B$.
\end{rem}

The next proposition and what follows are one of the main results and ideas in \cite{MR3943448} (see in particular Section 3).
\begin{prop}\label{prop: maxextension}
Let $(\Gamma, \alpha)$ be an abstract GKM graph, with axial function $\alpha\colon E\rightarrow
\ZZ^k$ and compatible connection $\nabla$. Then, with the above notation
and using the line Chern numbers of $(\Gamma,\alpha)$ for the
construction, the labelled graph
$(\Gamma,\gamma)$ is an abstract GKM graph and the connection $\nabla$ is compatible with the axial function $\gamma$.
Moreover the line Chern numbers of $(\Gamma,\gamma,\nabla)$ are the same as those of
$(\Gamma,\alpha,\nabla)$. 
Finally, there exists a linear map 
$\Phi\colon \ZZ^m\rightarrow \ZZ^k$ such that $\alpha(e) = \Phi(\gamma(e))$ for all $e\in E$.
\end{prop}

\begin{proof}
The axioms
\begin{enumerate}
\item $\gamma(e)=-\gamma(\overline{e})$
\item for any $p\in V$ and any $e,f\in E_p$, one has $\gamma(f)- \gamma(\nabla_e(f))=c_e(f)\gamma(e)\,$
\end{enumerate}
are clearly satisfied by construction. In order for $(\Gamma,\gamma)$ to be a GKM graph with compatible connection $\nabla$ and the desired line Chern numbers, it only remains to prove that $\gamma(f)$ and $\gamma(f')$ are linearly independent, for every pair of edges $f,f'\in E_p$, for every $p\in V$.
Since this holds for $\alpha$, this will be a direct consequence of the existence of the linear map $\Phi$ as in the statement of the proposition. Denote by $\alpha_i\colon E\rightarrow \ZZ$ the composition of $\alpha$ with the projection $\ZZ^k\rightarrow \ZZ$ onto the $i$th component, for every $i=1,\ldots,k$. We define analogously the components $\gamma_j$ of $\gamma$, for every $j=1,\ldots,m$. 
We have $\alpha_i\in L$ with $L$ as defined above. Since the $\gamma_j$'s are a basis of $L$, we find $a_{i1},\ldots,a_{im}\in \ZZ$ such that
\[\alpha_i=\sum_{j=1}^m a_{ij} \gamma_j.\]

Hence the map $\Phi\colon \ZZ^m\rightarrow \ZZ^k$ defined by the matrix $(a_{ij})$ satisfies $\alpha(e)=\Phi(\gamma(e))$ for any $e\in E$.
\end{proof}

In the construction of $\gamma$ we chose a basis for the group $L$. A different choice of a basis yields an a priori different GKM graph; however,
the map $\Phi$ in Proposition \ref{prop: maxextension} gives an isomorphism between the two GKM graphs.
Hence by slight abuse of language we make the following

\begin{defn}\label{def comb maximal}
We call the GKM graph $(\Gamma,\gamma)$ constructed above the \emph{maximal extension} of
$(\Gamma,\alpha ,\nabla)$. If $(\Gamma,\alpha)$ is GKM$_3$, then by
the maximal extension of $(\Gamma,\alpha)$ we mean the maximal extension with respect to
the unique connection (cf.\ Remark \ref{rem: gkm3connection}).
\end{defn}

\begin{cor}\label{virtual iso and max extensions}
 Let $(\Gamma,\alpha)$ and $(\Gamma', \alpha')$ be two GKM graphs
with fixed connections $\nabla$ and $\nabla'$ respectively. Then they are virtually
isomorphic if and only if their maximal extensions are isomorphic as abstract GKM graphs
through an isomorphism which preserves the connections. If both graphs
are GKM$_3$, then this is the case if and only if their maximal extensions are
isomorphic.
\end{cor}
\begin{proof}
Since the construction of the maximal extension only takes into account the line Chern numbers, it is clear that virtually isomorphic GKM graphs yield maximal extensions that are isomorphic as abstract GKM graphs. 
Conversely, any GKM graph isomorphism between maximal extensions,
which preserves the connections, also preserves in particular the
line Chern numbers of the maximal extensions (see Definition \ref{D: Iso of GKM graphs})
and hence yields a virtual isomorphism between the original GKM graphs.
Finally if $(\Gamma,\alpha)$ and $(\Gamma',\alpha')$ are GKM$_3$, then
so are their maximal extensions. In particular any isomorphism between the maximal
extensions will automatically preserve the unique connection (see Remark \ref{rem:
gkm3connection}).
\end{proof}

\begin{defn}\label{def combinatorially maximal}
Let $(\Gamma,\alpha)$ be an abstract GKM graph with fixed connection $\nabla$ and axial function
$\alpha\colon E\rightarrow\ZZ^k$. Furthermore let $c_*\colon (e,f)\mapsto c_e(f)$ be the
associated line Chern numbers, defined for every $e,f\in E_p$, for every $p\in V$. Let
$L(\Gamma,\nabla,c_*)$ be the lattice as constructed above. Then $(\Gamma,\alpha,\nabla)$
is called \emph{combinatorially maximal} if the components $\alpha_1,\ldots,\alpha_k$ of
$\alpha$ are a basis of $L(\Gamma,\nabla,c_*)$. For a GKM$_3$ graph we
also say that $(\Gamma,\alpha)$ is combinatorially maximal if it is combinatorially
maximal with respect to its unique connection. A GKM$_3$ space is called combinatorially
maximal if its GKM graph is.
\end{defn}

Clearly an abstract GKM graph is combinatorially maximal if and only if it is isomorphic to its maximal extension.

\begin{rem}\label{rem: combmax}
\begin{enumerate}\item
If the abstract GKM graph comes from a GKM space, being combinatorially maximal is a property of the graph only.
However, it gives the following geometric obstruction, namely: 
The torus action on the manifold cannot be extended effectively to any larger torus. 
The converse is not obvious, i.e.\ there might exist manifolds with a GKM action that is
maximal from a geometric point of view, but not combinatorially maximal. (See
\cite[Corollary 1.3]{MR3943448} as well as \cite[Section 4]{MR3943448}.) 
\item (Effective) Toric actions are combinatorially maximal: Fix a vertex $p$ and let $e_1,\ldots,e_n$ be the edges emanating from $p$.  
In \cite[Lemma 3.2]{MR3943448} it is proved that for any vertex $p$, the restriction of the projection $\ZZ^{E}\rightarrow \ZZ^{E_p}$ to the sublattice $L(\Gamma,\nabla,c_*)\subset \ZZ^{E}$ is an injection.
Composing with \[\ZZ^{E_p}\cong \ZZ^n,\qquad\beta\mapsto (\beta(e_1),\ldots,\beta(e_n))\] 
we obtain an injection $\psi\colon L(\Gamma,\nabla,c_*)\rightarrow \ZZ^{E_p}\cong \ZZ^n$. For the components $\alpha_i\in L$ of $\alpha$ we have $\psi(\alpha_i)=(\alpha(e_1)_i,\ldots,\alpha(e_n)_i)$. Note that the latter are the rows of the matrix $A$ whose columns are given by $\alpha(e_1),\ldots,\alpha(e_n)$. As the action is assumed to be toric one has $A\in \GL(n,\ZZ)$.
Consequently, $\psi(\alpha_1),\ldots,\psi(\alpha_n)$ is a basis of $\ZZ^n$ and it follows that the $\alpha_i$'s are a basis of $L(\Gamma,\nabla,c_*)$.

\item The injectivity of the map $L(\Gamma,\nabla,c_*)\rightarrow \ZZ^{E_p}$ described
	above is indeed the key point to prove that, for any GKM graph, the rank of the group
	$L(\Gamma,\nabla,c_*)$ is bounded above by the valency of the graph, namely the rank of
	$\ZZ^{E_p}$ (see \cite[Lemma 3.2]{MR3943448}).
	\item
Recently in \cite{GoertschesSolomadin},
	Goertsches and Solomadin proved that for a GKM graph of a Hamiltonian GKM$_{4}$ manifold, the rank of $L(\Gamma,\nabla,c_*)$ takes the maximum value, i.e.\ it equals the valency of $\Gamma$. In the literature, graphs of this maximal rank are often referred to as \emph{torus graph}. This partly answered a question of Masuda's, who asked whether any GKM$_4$ graph always extends to a torus graph. This dichotomy between GKM$_3$ actions and the toric world was recently supported in \cite[Theorem 2]{KurokiSolomadin}, where it was shown that maximal extensions of the GKM graphs of homogeneous spaces $G/H$ (with $G$ compact semi simple and $H$ of maximal rank) are either GKM$_2$, GKM$_3$, or torus graphs. As mentioned in the introduction, it is in general unclear whether the maximal extension of the GKM graph comes from an extension of the action on the manifold. In this regard, we note that while there are many interesting examples of GKM$_3$ manifolds whose action cannot be extended (due to being combinatorially maximal), we are not aware of a Hamiltonian GKM$_4$ manifold where the action does not extend to a toric action.
	
\end{enumerate}
\end{rem}

\subsection{Bounds on the maximal extension}\label{subseq: box}
\begin{center}
\textit{``Shall I describe it to you? Or would you like me to find you a box?''}\\ \vspace*{0.3cm}
{\hspace{5cm} --- Legolas Greenleaf}
\end{center}

In Section \ref{sec: CLN bound} we bounded the line Chern numbers of the GKM graph of a compact positive monotone Hamiltonian
 GKM$_3$ space. In this section we derive, from this, a bound on the moment image. First we prove
\begin{thm}\label{thm: Box}
Let $(\Gamma,\alpha,\nabla)$ be an $n$-valent GKM graph with fixed connection $\nabla$ and $\chi$ vertices. Assume that
for all $e,f\in E_p$ and all $p\in V$, we have $|c_e(f)|\leq \beta$ for some $\beta\in
\mathbb{N}$. Then there exists a constant $K=K(\beta,n,\chi)$ depending only on $\beta$,
$n$, and $\chi$, and a choice of maximal extension $\tilde\alpha\colon E\rightarrow \ZZ^k$
of $(\Gamma,\alpha,\nabla)$, for which 
\begin{equation}\label{inf bound alpha}
\| \tilde\alpha(e) \|_\infty\leq K\quad \text{for all  } e\in E\,.
\end{equation} 
The constant $K$ is given explicitly by
\begin{equation}\label{def K}
K=K(\beta,n,\chi)=\sqrt{\frac{(n+3)n\chi(\beta^2+2)^{{n\chi}}}{8}}\,.
\end{equation}
\end{thm}

\begin{rem}\label{rem: bound alpha}
\begin{itemize}
\item[(i)] We point out that the labels of a maximal extension are only well defined up to $\GL(k,\ZZ)$ transformations, which however do not preserve the $\|\cdot\|_\infty$ norm. 
Hence not every choice of maximal extension will satisfy the bound.
\item[(ii)] Secondly, obtaining such a bound without passing to the maximal extension is in general not possible, as it is shown by the example below.
\end{itemize}
\end{rem}

\begin{ex}\label{ex: long weights} Let $e_1,\ldots,e_n$ denote the standard basis of $\ZZ^n$.
Consider the $n$-dimensional cube in $\RR^n$ with vertices given at the points $(\pm 1, \ldots, \pm 1)$. This is exactly the moment map image
of the standard $T^n$ action on $(S^2)^n$, the latter endowed with symplectic form $\omega$ satisfying $c_1=[\omega]$. 
Its GKM graph corresponds to the union of edges and vertices of the cube above. Therefore, the weights at the fixed points
are exactly $(\pm e_1,\ldots, \pm e_n)$, depending on the vertex; here we identify $\ZZ^n$
with the dual lattice in $\mathfrak{t}^*$.
 
Let $L\in \ZZ^{(n-1)\times n}$ denote the matrix whose columns are the vectors
$e_1,\ldots,e_{n-1},m(e_1+\ldots+e_{n-1})$. Then the transpose $L^t$ defines an embedding
$T^{n-1}\rightarrow T^n$ which we use to restrict the $T^n$-action to a $T^{n-1}$ action.
The induced map on the dual Lie algebras is given by the matrix $L$ and the image of the
moment map of the restricted action is the projection of the cube under $L$. The labels
of edges emanating from a vertex become $e_1,\ldots,e_{n-1}$ and $m\sum_i e_i$
(with appropriate signs).

We claim that there is no automorphism of $\ZZ^{n-1}$ which maps the
weights of the $T^{n-1}$-action into the Euclidean ball of radius $C_m=\sqrt[n-1]{m}$. In
particular, for $m\to\infty$ we observe that a bound on the line Chern numbers,
which are always $0$ when using the standard connection, does not
enable a bound on the weights up to automorphism.
 
Let $A\in \SL(\ZZ,n-1)$. For $I\in \ZZ^{n\times (n-1)}$ with columns given by
$e_2,\ldots,e_n$ we have $\det(ALI)=\det(A)\det(LI)=\pm m$. It follows from Hadamards
inequality \cite{hadamard93, MR2978290} that there needs to be a column of $ALI$ of Euclidean norm $\geq C_m$. Since
these are also columns of $AL$, the claim holds.
\end{ex}

For the proof of Theorem \ref{thm: Box} we use two ingredients. The first is a number
theoretic principle generally known as Siegel's Lemma. The version below is an improved
version due to Bombieri and Vaaler (see \cite[Theorem 2]{MR707346}), which relies on techniques from the geometry of numbers.

\begin{thm}\label{thm: siegel}
Let $A\in \ZZ^{M\times N}$ with $N>M$. Let $D$ be the greatest common divisor among all $M\times M$ minors of $A$. Then there are $N-M$ linearly independent integral $x_i\in \ker A$ satisfying
\[\prod_{i=1}^{N-M}\|x_i\|_\infty\leq D^{-1}\sqrt{\det (AA^T)}.\]
\end{thm}

The second ingredient is the theory behind ``short bases'' of lattices. We review some
basic facts: Let $L\subset \RR^k$ be a lattice of rank $r$, for some $r\leq k$. Given a
basis $v_1,\ldots,v_r$ of $L$, we write the $v_i$ as the columns of a matrix $A\in
\RR^{k\times r}$. The determinant of the lattice is the number $\sqrt{\det(A^T A)}$, which
is independent of the choice of basis and hence an invariant of the lattice. We denote
this by $\det(L)$ and observe that it agrees with the volume of the fundamental
paralleliped of the lattice. Note that $A^{T}A = (\langle v_{i}, v_{j} \rangle)_{ij}$,
where $\langle \cdot, \cdot \rangle$ denotes the standard inner product of
$\mathbb{R}^{k}$. From Hadamard's inequality \cite[Theorem 7.8.1]{MR2978290} we have

\begin{equation}\label{eq: Hadamard}
	\det(L)\leq \sqrt{\prod_{i=1}^{r} \langle v_{i}, v_{i} \rangle} \leq\prod_{i=1}^r \|v_i\|_2,
\end{equation}

When it comes to finding short bases of lattices, an important invariant are
\emph{Minkowski's successive minima} $\lambda_1(L),\ldots,\lambda_r(L)$, where
$\lambda_i(L)$ is the minimal $\lambda\in \RR$ such that the closed Euclidean ball
$B_\lambda(0)$ contains $i$ vectors in $L$ that are $\QQ$-linearly independent (For the
original reference, see \cite{MR249269}). So in particular 
there are $r$ $\QQ$-linearly independent vectors of $L$ 
of Euclidean length $\leq \lambda_r(L)$. It is however not true that such a collection of
vectors is necessarily a $\ZZ$-basis of $L$ (they are a $\QQ$-basis of $L\otimes \QQ$). In
fact, starting in dimension $\geq 5$ there exist full rank lattices such that
$B_{\lambda_r(L)}(0)$ does not contain a $\ZZ$-basis of $L$ (see \cite[Sect. 5]{MR82513}).
Nonetheless lattice reduction techniques allow computations of bases that are short in an
appropriate (and very much non unique) sense. For our purposes we use the concept of
\emph{Korkine-Zolotarev bases}, as introduced in \cite{MR1509828}. These
types of bases, which exist for every lattice, are closely related to
Minkowski's successive minima and their norm can be bounded in terms of those. The
fundamental result that we use in this paper is due to Lagarias-Lenstra-Schnorr and is the
following

\begin{thm}[\cite{MR1099248}, Theorem 2.1]\label{thm: lattice reduction}
Let $L\subset \RR^k$ be a lattice of rank $r$. Then a Korkine-Zolotarev basis
$b_1,\ldots,b_r$ satisfies the following
property: For each $1\leq i\leq r$ 
\[\sqrt{\frac{4}{i+3}}\lambda_i(L)\leq \|b_i\|_2\leq \sqrt{\frac{i+3}{4}}\lambda_i(L).\]
\end{thm}

In particular, since $\lambda_i(L)\leq \lambda_{i+1}(L)$, one obtains a bound on the longest vector in a basis of $L$ in terms of $\lambda_r(L)$ and $r$.

\begin{proof}[Proof of Theorem \ref{thm: Box}]
We fix an identification $\ZZ^{E}\cong \ZZ^a$ by enumerating the edges. Let $A\in
\ZZ^{M\times a}$ be the matrix as constructed in \ref{rem: system of eq} and let $L=\ker
A\subset \ZZ^a$. A maximal extension of the original GKM graph was constructed in Section
\ref{subsec: max extensions} by a choice of basis of $L$. In particular, if we find a
basis of $L$ such that all $\infty$-norms of the basis vectors are bounded by a constant
$K$, then this constant also bounds the $\infty$-norms of all labels of the maximal
extension (see Remark \ref{rem: transpose business}). 

Therefore it remains to prove that we can find such a constant $K$ which only depends on
$\beta$, $n$ and the number of edges $a$, which is indeed $\frac{n}{2}\chi$. Let
$\Lambda\subset \ZZ^a$ be the lattice spanned by the rows of $A$. As a first step we
derive a bound on $\det(\Lambda)$. Let $v_1,\ldots,v_N$ be a basis of $\Lambda$ and set
$A'\in \ZZ^{N\times a}$ to be the matrix with rows $v_1,\ldots,v_N$. We note that $\ker
A'=\ker A=L$ and that $N< a$ since $L$ has positive rank. Choose any maximal subset of
linearly independent rows of $A$ and denote the resulting matrix by $\overline{A}$. By
construction, each row can be written as an integral linear combination of the $v_i$.
Hence we find $B\in \ZZ^{N \times N}$ such that $\overline{A} = B A'$. Furthermore $B$ has
nonzero determinant since $\overline{A}$ has full rank. Hence \[\sqrt{\det(
\overline{A}\cdot\overline{A}^T)}=|\det(B)|\sqrt{\det(A'(A')^T)}=|\det B| \det(\Lambda).\]
In particular since $B$ is integral we have $\det(\Lambda)\leq \sqrt{\det(
\overline{A}\cdot\overline{A}^T)}$. The rows of $A$ have Euclidean norm bounded by
$\beta_1:=\sqrt{\beta^2+2}$: indeed, they contain only $0$ entries except one line Chern
number and two entries $\pm 1$ (see Remark \ref{rem: system of eq}). It follows from
Hadarmard's inequality \cite[Corollary 7.8.3]{MR2978290} that
$$\sqrt{\det( \overline{A}\cdot\overline{A}^T)}\leq \beta_1^N$$ since the left hand side
is the volume of a fundamental domain of the lattice spanned by the columns of
$\overline{A}$.  Hence
\[\det(\Lambda) \leq \beta_1^N\leq \beta_1^a=:\beta_2.\]
Now by Theorem \ref{thm: siegel} applied to $A'$ (and replacing $D^{-1}$ in the theorem by $1$) we find a linearly independent set of integral vectors $x_{1},\ldots,x_{a-N}\in \ker A'=L$ such that
\[\prod_{i=1}^{a-N} \|x_i\|_\infty\leq \det(\Lambda).\]
Set $m:= a-N=\rk (L)$. Since the $x_i$ are integral we individually obtain $\|x_i\|_\infty\leq \det(\Lambda)$ and thus $\|x_i\|_2\leq \sqrt{a}\det(\Lambda)$. As the $x_i$ form a maximal linearly independent set in $L$, we deduce
\[\lambda_m(L)\leq \sqrt{a}\det(\Lambda).\]
Now by Theorem \ref{thm: lattice reduction} we may choose a basis $\gamma_1,\ldots,\gamma_m$ for $L$ with
\[\|\gamma_i\|_2\leq \sqrt{\frac{m+3}{4}}\lambda_m(L)\leq  \sqrt{\frac{(m+3)a}{4}} \det(\Lambda)\leq  \sqrt{\frac{(n+3)a}{4}}\beta_2 \]
where in the last inequality we have used that $m=\rk(L)$ is bounded above by the valency $n$ of $\Gamma$ (see Remark \ref{rem: combmax}). Therefore we can define the desired bound to be
\begin{equation}\label{def K1}
K=K(\beta,n,\chi):=\sqrt{\frac{(n+3)a}{4}}\beta_2=\sqrt{\frac{(n+3)n\chi(\beta^2+2)^{{n\chi}}}{8}}\,.
\end{equation}

\end{proof}

\begin{rem}\label{rem: Knchi}
Recall that by Corollary \ref{global bound} for a compact monotone Hamiltonian GKM$_3$ space of dimension $2n$ and Euler characteristic $\chi$, the line Chern numbers satisfy
\begin{equation} 
\label{def beta n}
 |c_e(f)|\leq \beta(n,\chi):= \begin{cases} 
24\chi & \text{if  }n=3\\
(n-3)\chi \cdot\frac{n(n+1)^2}{2}+2 & \text{if  } n>3
\end{cases}\nonumber
\end{equation}
Inserting $\beta=\beta(n,\chi)$ into $K(\beta,n,\chi)$ from \eqref{def K} we obtain a constant $K(n,\chi)$ only depending on $n$ and $\chi$.
\end{rem}

Consider now a compact $2n$-dimensional symplectic manifold, endowed with a Hamiltonian
action of $T=T^k$ with moment map $\mu\colon M \to \mathfrak{t}^*\cong
\RR^k$. 
We define the \emph{diameter} $d_\mu$ of the moment map $\mu$ as being the
$\|\cdot\|_\infty$-length of the longest segment connecting $\mu(p)$ to $\mu(q)$ in
$\mu(M)\subset \RR^k$, for all $p,q\in M$, namely 
$$
d_\mu:=\max\{\|\mu(p)-\mu(q)\|_\infty, \;p,q,\in M\}\,.
$$
\begin{cor}\label{cor: box}
Let $(M,\omega,\mu)$ be a compact positive monotone Hamiltonian $\text{GKM}_3$ space of dimension $2n$ and Euler characteristic $\chi$.
Assume that $c_1=[\omega]$ and that the action is combinatorially maximal.
Then there exist a constant $K'=K'(n,\chi)$, explicitly given by \eqref{bound
diameter}, such that, after potentially precomposing the action by an automorphism of
$T$, we have
\begin{equation}\label{diameter}
d_\mu\leq K'(n,\chi)\,.
\end{equation}
\end{cor}

\begin{proof}
Let $(\Gamma,\alpha)$ be the GKM graph of $(M,\omega,\mu)$ with $\Gamma=(E,V)$. Since the symplectic manifold is positive monotone, 
Corollary \ref{global bound} gives a constant $\beta=\beta(n,\chi)$ such that for any adjacent $e,f\in E$ we have $|c_e(f)|\leq \beta$
 
From Theorem \ref{thm: Box} and Remark \ref{rem: Knchi} we infer that there is a constant $K=K(\beta,n,\chi)=K(n,\chi)$ such that 
the weights of the a maximal extension $(\Gamma,\tilde{\alpha})$ satisfy
\begin{equation}\label{bound alpha }
\|\tilde\alpha(e)\|_\infty\leq K(n,\chi)\quad \text{for all  } e\in E\,.
\end{equation}
As the $T$-action is assumed to be combinatorially maximal, there is an isomorphism $(\Gamma,\alpha)\cong (\Gamma,\tilde{\alpha})$ of GKM-graphs. Changing the original action by a suitable automorphism of $T$, we may assume $\alpha=\tilde{\alpha}$.

By the weight sum formula \eqref{weight sum formula} we have $\mu(p)=-\sum_{i=1}^n\alpha_i(p)$ for every $p\in M^T$, where
$\alpha_1(p),\ldots,\alpha_n(p)$ are the isotropy weights at $p$. Hence, by \eqref{bound alpha },
\[
\|\mu(p)\|_\infty\leq \sum_{i=1}^n\|\alpha_i(p)\|_\infty\leq n\,K(n,\chi)\quad\text{for all }p\in M^T\,.
\]
Consequently, for any $p,q\in M^T$
\begin{equation}\label{bound diameter}
\|\mu(p)-\mu(q)\|_\infty\leq 2n\,K(n,\chi)=: K'(n,\chi).
\end{equation}
Convexity of the moment image implies that this bound actually holds for all $p,q\in M$.
\end{proof}

\begin{rem}\label{rem better bound}

The constant $K(n,\chi)$ in \eqref{bound alpha } was defined in \eqref{def K}, and by Corollary \ref{global bound} one can take 
\[\beta = \chi\cdot \frac{(n-3)n(n+1)^2}{2}+2\]
starting from $n\geq 4$.

Therefore we obtain
$$
\|\mu(p)-\mu(q)\|_\infty\leq 2n \sqrt{\frac{(n+3)n\chi(\beta^2+2)^{n\chi}}{8}}\,.
$$

One can simplify this bound as follows. For $n\geq 4$, we have
$$
\beta<\frac{\chi n^4}{2}-1,
$$
and hence
$$
\beta^2+2<\frac{\chi^2n^8}{4}.
$$
The latter inequality also holds for $n=3$, in which case $\beta=24\chi$ by Remark (\ref{rem: Knchi}). Moreover, since
$$
\frac{(n+3)n\chi}{8}\geq 1
$$
for $n\geq 3$, its square root is bounded above by the quantity itself. Consequently, for all $n\geq 3$, the constant $K'(n,\chi)$ in Corollary \ref{cor: box} can be taken to be
\begin{equation}\label{expl def chi}
K'(n,\chi)=\frac{n^2(n+3)\chi}{4}
\left(\frac{n^4\chi}{2}\right)^{n\chi}.
\end{equation}
\end{rem}

\section{Chern numbers and cobordism types}

\subsection{On the (equivariant) cobordism type}
The following lemma is key to define the combinatorial analogue of equivariant Chern numbers.
\begin{lemma}\label{ABBV integers}
Let $(\Gamma,\alpha)$ be an abstract GKM graph that is $n$-valent and $f$ a class in
$\mathcal{H}^*_T(\Gamma, \alpha;\ZZ)$ of degree $2l$. If $l\geq n$ then
\end{lemma}
\begin{equation}\label{localization}
f_*:=\sum_{v\in V} \frac{f(v)}{\prod_{e\in E_v}\alpha(e)}\in H^{2(l-n)}(BT;\ZZ)\,.
\end{equation}
\begin{proof}
The above result is proved in \cite[Theorem 2.2]{MR1701922} 
for complex coefficients. However the proof applies \emph{ad verbatim} to integral coefficients.
\end{proof}
We note that the compatibility condition \eqref{def ce(f)} satisfied by the weights
 at the endpoints of an edge $e$ ensures that the map
 \begin{equation}\label{comb chern}
 V \ni v\mapsto \sigma_i(\alpha_1,\ldots,\alpha_n)\in H^{2i}(BT;\ZZ)
 \end{equation}
 satisfies the compatibility condition in Definition \ref{D: graph cohomology}, where $\sigma_i$ denotes the $i$-th elementary symmetric polynomial, and 
 $\alpha_1,\ldots,\alpha_n$ are the labels of the edges with initial point given by $v$. Therefore the class defined in \eqref{comb chern} belongs to 
 $\mathcal{H}^{2i}_T(\Gamma, \alpha;\ZZ)$; we refer to it as the \emph{combinatorial equivariant $i$-th Chern class of }$(\Gamma,\alpha)$ and denote it by $c^T_i$, for all $i=0,\ldots,n$.
 
For any $m$-tuple of integers $i_1,\ldots,i_m\geq 1$ satisfying $\sum_j i_j\geq n$, using \eqref{localization} it is also possible to define
 the \emph{combinatorial equivariant Chern number} $c^T_{i_1,\ldots,i_m}[\Gamma]$ as 
 \begin{equation}\label{comb chern numbers}
 c^T_{i_1,\ldots,i_m}[\Gamma]:=(c^T_{i_1}\cdots c^T_{i_m})_*\,\in H^{2(\sum_j  i_j-n)}(BT;\ZZ)
 \end{equation}

Note that if $\sum_j i_j =n $, then $c^T_{i_1,\ldots,i_m}[\Gamma]\in H^0(BT)=\mathbb{Z}$ is an integer. In this case we also write $c_{i_1,\ldots,i_m}[\Gamma]$ and call it the \emph{combinatorial Chern number}.
The above definitions are inspired by the geometric counterparts. Indeed,
if $(\Gamma,\alpha)$ is the GKM graph of a GKM space $(M,J,T)$, then
\begin{itemize}
\item[(a)] Given a class $f\in H^*_T(M;\ZZ)$ and its restriction $f(v)$ at the fixed points $v\in M^T$,  
the polynomial in \eqref{localization} corresponds exactly to the evaluation of $f$ on the
fundamental cycle $[M]$, namely $f[M]$; alternatively, when using the
	Cartan model for 
equivariant cohomology, this can be described as integration over $M$; 
\item[(b)] The combinatorial equivariant $i$-th Chern class is exactly the restriction to the fixed points
of the equivariant $i$-th Chern class of the (equivariant) tangent bundle of $M$;
\item[(c)] By (a) and (b) the combinatorial equivariant Chern numbers of $(\Gamma,\alpha)$ are the equivariant Chern numbers of $(M,J,T)$.
We denote the latter by $ c^T_{i_1,\ldots,i_m}[M]$, for any $m$-tuple of integers $i_1,\ldots,i_m\geq 1$ satisfying $\sum_j i_j\geq n$.
\end{itemize}
 In the rest of the section the coefficient ring of $H^*(BT)$ is assumed to be always $\ZZ$.
\begin{lem}\label{lem: restrict chern numbers}
Let $T^k$ be a $k$-dimensional compact torus and
let $(\Gamma,\alpha)$ and $(\Gamma',\alpha')$ be abstract GKM graphs with labels in $H^2(BT^r)$ and $H^2(BT^k)$ respectively. Assume that there is a graph isomorphism $\varphi\colon \Gamma\rightarrow \Gamma'$ between them together with a linear map $\Phi\colon H^2(BT^r)\rightarrow H^2(BT^k)$ satisfying $\alpha'(\varphi(e))=\Phi(\alpha(e))$ for any $e\in E$. Then the multiplicative extension $\Phi\colon H^*(BT^r)\rightarrow H^*(BT^k)$ maps the combinatorial equivariant Chern numbers of $(\Gamma,\alpha)$ to those of $(\Gamma',\alpha')$.
\end{lem}

\begin{proof}
This is an immediate consequence of \eqref{localization}, the definition of (combinatorial) equivariant Chern classes \eqref{comb chern}, and \eqref{comb chern numbers}.
\end{proof}

\begin{defn}\label{defn: virtually equiv chern number}
Let $(M,J,T^k)$ and $(N,J',T^l)$ be $n$-dimensional almost complex manifolds with compatible actions of the tori $T^k$ and $T^l$. We say that $(M,J,T^k)$ and $(N,J',T^l)$ have 
\emph{virtually equivalent equivariant Chern numbers} if there exists $r\geq k,l$, homomorphisms $\iota_1\colon T^k\rightarrow T^r$, $\iota_2\colon T^l \to T^r$, and for any $i_1+\ldots+i_m\geq n$ a polynomial 
$c^T_{i_1,\ldots,i_m}\in H^*(BT^r)$  such that
\[c^T_{i_1,\ldots,i_m}[M] = \iota_1^*(c^T_{i_1,\ldots,i_m})\quad \text{and}\quad
c^T_{i_1,\ldots,i_m}[N]=\iota_2^*(c^T_{i_1,\ldots,i_m}).\]

\end{defn}

\begin{prop}\label{fin many equiv Chern numbers}
For fixed $\chi,n\in \NN$, up to virtual equivalence there are only finitely
many collections of polynomials that can arise as the
collection of equivariant Chern numbers
of a compact positive monotone Hamiltonian GKM$_3$ space of dimension $2n$ and Euler characteristic $\chi$.
\end{prop}
\begin{proof}
By Corollary \ref{cor: main} we only need to prove that virtually isomorphic GKM graphs have virtually equivalent equivariant Chern numbers. Let $M$ be a $T^k$-manifold and $N$ be a $T^l$-manifold with virtually isomorphic GKM graphs $(\Gamma_M,\alpha_M)$ and $(\Gamma_N,\alpha_N)$. 
From Proposition \ref{prop: maxextension} we get that the maximal extensions
$(\Gamma_M,\gamma_M)$ and $(\Gamma_N,\gamma_N)$ are isomorphic. Since both have labels in
a free Abelian group of some rank $r$, we may identify the latter with $H^2(BT^r)$. We obtain
projections of GKM graphs as in the requirements of Lemma \ref{lem:
restrict chern numbers}
\[(\Gamma_M,\alpha_M)\xleftarrow{(\mathrm{id}_{\Gamma_M},\Phi_M)} (\Gamma_M,\gamma_M)\xrightarrow{(\varphi,\Psi)} (\Gamma_N,\gamma_N)\xrightarrow{(\mathrm{id}_{\Gamma_N},\Phi_N)} (\Gamma_N,\alpha_N)\]
where $\Phi_M\colon H^2(BT^r)\rightarrow H^2(BT^k)$, $\Phi_N\colon H^2(BT^r)\rightarrow H^2(BT^l)$ are (surjective) linear maps, $\Psi$ is an automorphism of $H^2(BT^r)$ and $\varphi$ is some graph isomorphism $\Gamma_M\rightarrow \Gamma_N$. We point out that any linear map such as $\Phi_M$ is indeed induced by a homomorphism $T^k\rightarrow T^r$. Now the virtual equivalence of the equivariant Chern numbers of $M$ and $N$ follows from applying Lemma \ref{lem: restrict chern numbers} to all three of the above arrows.
\end{proof}

\begin{thm}\label{thm: bordism} $\;$
\begin{itemize}[leftmargin=*]
\item[(i)] For every $\chi,n\in \NN$, there are finitely many complex cobordism classes of compact positive monotone symplectic manifolds of dimension $2n$ and Euler characteristic $\chi$ admitting a Hamiltonian GKM$_3$ action. 
\item[(ii)] For every $\chi,n\in \NN$, up to pulling back actions along automorphisms of $T$, there are finitely many $T$-equivariant complex bordism classes of compact positive monotone symplectic manifolds of dimension $2n$ and Euler characteristic $\chi$ with a combinatorially maximal Hamiltonian GKM$_3$ action of $T$.
\end{itemize}
\end{thm}

\begin{proof}
In order to prove $(i)$ note that, for a multi index summing up to the dimension of the manifold, the equivariant
Chern number of a $T$-manifold will be an element of $H^0(BT)\cong \ZZ$ and, under this
isomorphism, is equal to the non-equivariant Chern number. Since any homomorphism
$T^k\rightarrow T^r$ induces the identity $\ZZ\cong H^0(BT^r)\rightarrow H^0(BT^k)\cong
\ZZ$, it follows that, if two manifolds with torus actions $M$ and $N$ have virtually
equivalent equivariant Chern numbers in the sense of Definition \ref{defn: virtually equiv
chern number}, then they have identical non-equivariant Chern numbers. Therefore Proposition
\ref{fin many equiv Chern numbers} implies that only finitely many Chern numbers occur. Now part $(i)$ follows from
results of Milnor \cite{MR119209}. 

To prove part $(ii)$ we observe that Corollaries \ref{cor: main} and \ref{virtual iso and max extensions}
imply that there are
 only finitely many isomorphism types of maximal
extensions of the corresponding GKM graphs. Since the actions in $(ii)$ are assumed to be
combinatorially maximal, the GKM graphs of the manifolds in question indeed belong to one
of these finitely many isomorphism types. Given two GKM $T$-manifolds with isomorphic GKM
graphs, there is a linear automorphism of $H^2(BT)$ transforming the weights of one action
into the weights of the other. This determines an automorphism $\phi$ of $T$ such that,
after pulling back one of the actions along $\phi$, the weights of the actions are indeed
the same (after identifying the underlying graphs in suitable fashion). In particular the
associated equivariant Chern numbers are identical. Now $(ii)$ follows from 
\cite[Theorem H.4]{MR1929136}.
\end{proof}

\begin{ex} \label{ex: eqbordism}
Part $(ii)$ of Theorem \ref{thm: bordism} is indeed false when the condition on the actions being combinatorially maximal is dropped. This is due to the fact that restricting a single $T^n$-action to different $k$-dimensional subtori can produce $T^k$-manifolds which are not equivariantly bordant, even up to automorphisms of $T^k$. We give a concrete example: Consider $M=\CC P^4$ together with the standard action whose moment polytope (w.r.t. a suitable moment map) is the convex hull of $0$ and the standard basis vectors $w,x,y,z$ of $\mathfrak{t}^*\cong \RR^4$. Then one computes the equivariant Chern number
\[ c_{4,2}[M]=\sum_{p\in M^T} c_2(p) = 6(w^2+x^2+y^2+z^2)-3(wx+wy+wz+xy+xz+yz).\]
We view this quadratic form $q$ as
\[q=\begin{pmatrix}
w&x&y&z
\end{pmatrix} \frac{3}{2}\begin{pmatrix}
4 & -1& -1 & -1\\ -1& 4& -1& -1\\ -1& -1& 4& -1\\ -1 & -1& -1& 4
\end{pmatrix}
\begin{pmatrix}
w\\x\\y\\z
\end{pmatrix}.\]
We denote the middle matrix (including the scalar) by $M_q$. For some $k\in \ZZ$, we pull back the action along the homomorphism $\phi_k\colon T^3\rightarrow T^4$ determined by the matrix
\[A_k=\begin{pmatrix}
1 & 0 & 0\\ 0 & 1 &0\\ 0&0&1\\ 1 & 1 & k
\end{pmatrix}\]
and note that the result will be a Hamiltonian GKM$_3$-action when $k\neq 0$ (with weights at the origin given by the rows of $A_k$).
The induced map $\varphi_k\colon H^*(BT^4)\cong \ZZ[w,x,y,z]\rightarrow \ZZ[x,y,z]\cong H^*(BT^3)$ is determined by its values on $H^2(BT^4)$ where it is given by the matrix $A_k^t$ with respect to the bases $w,x,y,z$ and $x,y,z$.
The equivariant Chern number $c_{4,2}$ is natural with respect to the restriction along
$\varphi_{k}$ and transforms into
\[\varphi_k(q)=\begin{pmatrix}
x&y&z
\end{pmatrix}
M_{\varphi(q)}
\begin{pmatrix}
x\\y\\z
\end{pmatrix}\]
with \[M_{\varphi_k(q)}= A_k^t M_q A_k = \frac{3}{2}\begin{pmatrix}
6 & 1 & 3k-2\\ 1 & 6 & 3k-2\\ 3k-2& 3k-2 & 4k^2- 2k +4
\end{pmatrix}.\]
We claim that the $\varphi_k(q)$ are pairwise distinct, even up to automorphism of $H^*(BT^3)$, for infinitely many values of $k$. This then indeed shows that the resulting $T^3$-manifolds lie in pairwise distinct equivariant bordism classes even up to automorphisms of $T^3$. To prove the claim we note that pulling back $\varphi(q)$ along an automorphism of $H^*(BT^3)$ will yield a polynomial of the form
\[\varphi_k(q)=\begin{pmatrix}
x&y&z
\end{pmatrix}
B^t
 M_{\varphi(q)}
B
\begin{pmatrix}
x\\y\\z
\end{pmatrix}\]
for some $B\in \GL(3,\ZZ)$. Hence the determinant of the symmetric matrix which represents the quadratic forms is invariant under automorphisms of $H^*(BT^3)$. However $\det M_{\varphi(q)}$
takes infinitely many distinct values when varying $k$.
\end{ex}

\subsection{Finiteness of reflexive GKM$_k$ graphs}\label{reflexive GKM graphs}

We recall that a \emph{reflexive GKM graph} is the GKM graph of a compact positive monotone Hamiltonian GKM space $(M,\omega,\psi)$ (with $c_1=[\omega]$),
see \cite[Definition 5.13 and Proposition 5.15]{MR3695881}. These graphs can be visualized in $\mathfrak{t}^*$ using the moment map $\mu$ (see also Remark \ref{R: Existence of GKM graph}); 
the image of the GKM graph in $\mathfrak{t}^*$ is denoted by $\mu(\Gamma)$.
Note that here we do not insist that the moment map is injective on the fixed point set $M^T$. 

As a natural generalization, we define a \emph{reflexive GKM$_k$ graph} to be the GKM graph of a compact positive monotone Hamiltonian GKM$_k$ space $(M,\omega,\psi)$ (with $c_1=[\omega]$).
The proof of Theorem \ref{thm: bordism} (ii) leads immediately to the proof of Corollary \ref{cor reflexive graphs intro} in the introduction, stating that there
are only finitely many combinatorially maximal reflexive GKM$_3$ graphs of valency $n$ and with $\chi$ vertices. 
This result can be regarded as a generalization of the analogous fact for reflexive polytopes, where however
the number of vertices $\chi$ has an upper bound depending only on $n$. 
We highlight that for low valencies of the GKM graph (or equivalently for low dimensions of $M$), one has an upper bound on the number of vertices $\chi$ (corresponding to the Euler characteristic of $M$) in the following setting:
\begin{itemize}
\item If $n=2$ then the GKM action on the associated $(M,\omega)$ is indeed toric, and the corresponding GKM graph, embedded in the dual of the Lie algebra $\mathfrak{t}^*$ using the moment map, is therefore a smooth reflexive polytope. This implies that $\chi \leq 6$.
\item If $n=3$, it follows from the results contained in \cite[Section 5]{MR3695881} (see in particular Table 2), that the second Betti number $b_2$ is bounded above by 7, therefore implying that
$\chi=2+2b_2\leq 16$ (see also \cite[Lemma 4.9]{MR5005925}).
\end{itemize}
For $n\geq 4$ one also obtains upper bounds for $\chi$ in the following setting.
\begin{prop}\label{upper bound chi k0}
Let $(M,\omega,\mu)$ be a compact positive monotone Hamiltonian GKM space of dimension $2n$ and Euler characteristic $\chi$. Assume that the even Betti numbers are unimodal, namely 
$$1=b_0\leq b_2\leq \cdots \leq b_{2\lfloor n/2 \rfloor}.$$
If the index $k_0$ of $(M,\omega)$ satisfies $k_0\geq n-2$, then there is an upper bound for $\chi$ depending only on $n$ and $k_0$. 
\end{prop}
\begin{proof}
By \cite[Corollary 1.1]{MR3665359} applied to a generic component of the moment map, we have that $k_0\leq n+1$. Then the result
follows directly from \cite[Corollary 5.8]{MR3695881} and the fact that a Hamiltonian GKM space admits naturally a toric 1-skeleton (see also \cite[Section 4.1]{MR3695881}).
\end{proof}

\subsection{A quantitative bound for the Chern numbers}
Our methods also lead to a quantitative bound of the Chern numbers of a positive monotone Hamiltonian GKM$_3$ space. This is achieved through the following

\begin{prop}\label{prop: chernbound}
Let $(\Gamma,\alpha)$ be an $n$-valent GKM graph with $\chi$ vertices and edge set $E$, where the axial function $\alpha\colon E\rightarrow \ZZ^k$ satisfies $\|\alpha(e)\|_\infty \leq K$, for some $K>0$ and all $e\in E$. 
Let $c^T_{i_1,\ldots,i_m}[\Gamma]\in  H^{0}(BT)=\ZZ$ be the combinatorial Chern number corresponding to the decomposition $n=i_1+\ldots+i_m$. Then we have
\[|c^T_{i_1,\ldots,i_m}[\Gamma]|\leq\chi\cdot (k\,K (K+1)^{k-1})^n\prod_{j=1}^m {n\choose i_j}.\]
\end{prop}
\begin{proof}
Let $v_1,\ldots, v_\chi$ be the vertices and $\alpha_{i,1},\ldots,\alpha_{i,n}$ be the weights at $v_i$. Recall that for any $l$ we have $c_l^T(v_i)=\sigma_{l}(\alpha_{i,1},\ldots,\alpha_{i,n})$ where $\sigma_l$ is the $l$-th elementary symmetric polynomial. Hence
\[c^T_{i_1,\ldots,i_m}[\Gamma] = \sum_{i=1}^{\chi} \frac{\prod_{j=1}^m \sigma_{i_j}(\alpha_{i,1},\ldots,\alpha_{i,n})}{\alpha_{i,1}\cdot\ldots\cdot\alpha_{i,n}},\]
which is a priori an element of the field of fractions of $H(BT;\ZZ)$, but it simplifies to an integer by Lemma \ref{ABBV integers}. For some $X\in \mathfrak{t}\cong \RR^k$ we denote by $\alpha_{i,j}(X)\in \RR$ the evaluation of $\alpha_{i,j}\in H^2(BT;\ZZ)\cong \ZZ_\mft^*$ extended to an element of $\mft^*$. When viewing $\alpha_{i,j}\in \ZZ^k$ this means $\alpha_{i,j}(X)=\langle \alpha_{i,j},X\rangle$ and thus $|\alpha_{i,j}(X)|\leq k \|\alpha_{i,j}\|_\infty \cdot \|X\|_\infty\leq kK\|X\|_\infty$. Consequently
\[|\sigma_{i_j}(\alpha_{i,1},\ldots,\alpha_{i,n})(X)|\leq {n \choose i_j}(kK\|X\|_\infty)^{i_j}.
\]

Assume for the moment that we have $X\in \RR^k$ such that $|\alpha_{i,j}(X)|\geq 1$ for all $i,j$. Seeing $c^T_{i_1,\ldots,i_m}[\Gamma]$ as a constant rational function on the set of such $X$'s, we obtain
\begin{align*}
|c^T_{i_1,\ldots,i_m}[\Gamma]| &= \left|c^T_{i_1,\ldots,i_m}[\Gamma] (X)\right|\\
&\leq \sum_{i=1}^\chi \frac{\prod_{j=1}^m |\sigma_{i_j}(\alpha_{i,1},\ldots,\alpha_{i,n})(X)|}{|\alpha_{i,1}(X)\cdot\ldots\cdot\alpha_{i,n}(X)|}\\
&\leq \sum_{i=1}^\chi \prod_{j=1}^m |\sigma_{i_j}(\alpha_{i,1},\ldots,\alpha_{i,n})(X)|\\
&\leq \chi\cdot (k\,K\|X\|_\infty)^n\prod_{j=1}^m {n\choose i_j}.
\end{align*}
To finish the proof it remains to give a choice for $X$ which is not orthogonal to any of the weights. We set
\[X= (1,(K+1),(K+1)^2,\ldots, (K+1)^{k-1}).\]
Indeed, consider any $\alpha =(\alpha_1,\ldots,\alpha_k)\in \ZZ^k$ with $\|\alpha_i\|_{\infty}\leq K$ and let $j$ be the maximal index, such that $\alpha_j\neq 0$. Then
\begin{align*}
|\langle\alpha,X\rangle| &=\left|\sum_{i=0}^j \alpha_i (K+1)^{i-1}\right|\\ &\geq |\alpha_j| (K+1)^{j-1}-\sum_{i=1}^{j-1} |\alpha_i|(K+1)^{i-1}\\ &\geq (K+1)^{j-1} - \sum^{j-1}_{i=1} K(K+1)^{i-1} =1.
\end{align*}
We conclude by observing that as $K> 0$ we have $\|X\|_\infty = (K+1)^{k-1}$.
\end{proof}

\begin{rem}
Geometrically, for rational $K$ the choice of $X$ in the above proof corresponds to that of a generic circle $S^1\subset T^k$ such that one still has isolated fixed points after restricting the action to $S^1$.
 For any individual GKM manifold one usually finds such a generic circle, which is more optimal than the one in Proposition \ref{prop: chernbound}. However in order to obtain the above statement, which does not depend on the individual weights, it is necessary to find a circle which is simultaneously generic for all GKM manifolds with ``small'' weights.
\end{rem}

\begin{cor}\label{thm: quantitative bound}
Let $(\Gamma,\alpha,\nabla)$ be an $n$-valent GKM graph with $\chi$ vertices and fixed connection. Let $\beta\in \NN$ such that the line Chern numbers satisfy $|c_e(f)|\leq \beta$.
Then for each decomposition $n=i_1+\ldots+i_m$, the corresponding combinatorial Chern number satisfies
\begin{equation}\label{eq def L}
|c_{i_1,\ldots,i_m}[\Gamma]|\leq L(\beta,n,\chi):= \chi\cdot (n\,K (K+1)^{n-1})^n\prod_{j=1}^m {n\choose i_j}
\end{equation}
where $K$ is the constant $K(\beta,n,\chi)$ from Theorem \ref{thm: Box} equation \eqref{def K}
\end{cor}

\begin{proof}

Theorem \ref{thm: Box} implies that there is a maximal extension $(\Gamma,\tilde{\alpha})$, $\tilde{\alpha}\colon E\rightarrow \ZZ^k$, of the GKM graph such that the weights $\tilde{\alpha}(e)$ satisfy $\|\tilde{\alpha}(e)\|_\infty\leq K=K(\beta,n,\chi)$, where the latter is defined in equation \eqref{def K}. Proposition \ref{prop: chernbound}, together with the fact that $k\leq n$ (see Remark \ref{rem: combmax}), implies
\[c^{\widetilde{T}}_{i_1,\ldots,i_m}[\Gamma]\leq\chi\cdot (n\,K (K+1)^{n-1})^n\prod_{j=1}^m {n\choose i_j}\]
where the left hand side is the combinatorial Chern number of the maximal extension $(\Gamma,\tilde{\alpha})$ and $\widetilde{T}$ an abstract torus such that 
its dual lattice is $\ZZ^k$. The GKM graph $(\Gamma,\alpha)$ of the given
$T$-action on $(M,\omega)$ is a restriction of $(\Gamma,\tilde{\alpha})$ in the sense of
Lemma \ref{lem: restrict chern numbers}. Thus by the lemma,
$c^{\widetilde{T}}_{i_1,\ldots,i_m}[\Gamma]$ restricts to the respective equivariant Chern
number of $(\Gamma,\alpha)$.

\end{proof}

\begin{cor}\label{cor: quantitativeboundgeometric}
Let $(M,\omega)$ be a compact positive monotone Hamiltonian GKM$_3$ space of dimension $2n$ and Euler Characteristic $\chi$. Then for $K(n,\chi)$ as in Remark \ref{rem: Knchi}, we have
\begin{equation}\label{eq def L geometric}
|c_{i_1,\ldots,i_m}[M]|\leq L(n,\chi):=\chi\cdot (n\,K (K+1)^{n-1})^n\prod_{j=1}^m {n\choose i_j}.
\end{equation}
In particular, there exists a polynomial $P(n,\chi)$ such that for any $(M,\omega)$ as above all of its Chern numbers satisfy \[|c^T_{i_1,\ldots,i_m}[M]| \leq P(n,\chi)^{n^3\chi}.\]
\end{cor}
\begin{proof}
The first inequality follows directly from Corollary \ref{thm: quantitative bound} and Remark \ref{rem: Knchi} and it remains only to justify the existence of $P(n,\chi)$. Taking into account the explicit definition of $K(\beta,n,\chi)$ from Equation \eqref{def K} and $\beta(n,\chi)$ from Equation \eqref{def beta n} one sees that there is a polynomial $P_1(n,\chi)$ such that $K\leq P_1(\chi,n)^{n\chi}$. Hence for some other polynomial $P_2(n,\chi)$ we have
\[\chi\cdot (n\,K (K+1)^{n-1})^n\leq P_2(n,\chi)^{n^3\chi}.\]
Clearly the product over the binomial coefficients has an upper bound of the same form (for instance it is bounded by $n^{n^2}$) and consequently so does $c_{i_1,\ldots,i_m}[M]$.
\end{proof}

In the special case of $c_1^n[M]$ we obtain

\begin{cor}\label{bound on the volume}
With the notation in Corollary \ref{cor: quantitativeboundgeometric} one has
\[ \int_M \omega^n=c_1^n[M]\leq \chi (n^2\,K (K+1)^{n-1})^n.\]
\end{cor}

\begin{rem}\label{fine panov}
By the construction in \cite[Section 5]{MR2679581}, it is not hard to see that, in each dimension
$2n\geq 12$, it is possible to 
produce a sequence of non-simply connected positive monotone symplectic manifolds $(M_k,\omega_k)$
such that both their volumes $c_1^n[M_k]$ and Euler characteristics tend to infinity as $k\to \infty$. This is in contrast with the Fano case, 
as by the celebrated Koll\'ar-Miyaoka-Mori theorem \cite[Theorem 0.2]{MR1189503}, for each $n$ there exists a constant 
$C(n)$ such that $c_1^n[M]\leq C(n)$. However we wonder whether a bound in the spirit of that obtained in Corollary
\ref{bound on the volume} is possible.
\end{rem}

\section{Applications to smooth reflexive polytopes}\label{sec application reflexive}
In this section we analyze some consequences of our theorems for reflexive polytopes of dimension $n\geq 3$ (as in dimension
$n=2$ there are just five well-known reflexive polygons, up to lattice transformations).

Suppose that $(M,\omega,\mu)$ is a compact positive monotone Hamiltonian GKM$_n$ space of dimension $2n$, therefore
a symplectic toric manifold of dimension $2n$.
Then the image of the moment map $\mu(M)=:\Delta$ is a \textit{smooth reflexive polytope} of dimension $n$ and the GKM graph coincides
with the set of vertices and edges of $\Delta$ with the following labels: If $e$ is an edge going from $p$ to $q$, where
$p,q\in \RR^n$ are vertices of $\Delta$, then $\alpha(e)$ is the primitive vector in $\ZZ_{\mathfrak{t}}^*\simeq \ZZ^n$ in the direction of $e$
(see \cite{MR3695881}). Moreover, for a smooth polytope of dimension $n$, the
corresponding GKM graph has a standard, unique connection $\Delta$: For every pair of
distinct edges $e,f$ adjacent to the vertex $p$, consider the two
dimensional face $\mathcal{F}$ of $\Delta$ containing $e$ and $f$. If $e$ joins the
vertices $p$ and $q$, let $f'$ be the unique edge belonging to $\mathcal{F}$ and
adjacent to $q$;  we define $\nabla_e(f)$ to be $f'$. By
\cite[Proposition 2.4]{MR3695881}, this connection is compatible with the axial function
defined above (here the line Chern number $c_e(f)$ is exactly the integer $a_i$ in
\cite[Proposition 2.4]{MR3695881}). 
We refer to the line Chern numbers of this (unique) connection as the \emph{line Chern numbers of the smooth polytope} $\Delta$. 

Conversely, a smooth reflexive polytope $\Delta$ of dimension $n$ is always the moment map image of a compact positive monotone
symplectic toric manifold $(M_\Delta,\omega,\mu)$ of dimension $2n$ (see \cite[Proposition 3.10]{MR3695881}) and the latter is a Hamiltonian GKM$_n$ space.

With these ingredients at hand, we are able to give the following

\begin{proof}[Proof of Corollary \ref{cor reflexive1}]\label{proof of cor intro}
Observe that a GKM$_n$ action is combinatorially maximal (see Remark \ref{rem: combmax} (2)).
Moreover, the number of vertices $|V|$ of $\Delta$ corresponds exactly to the number of fixed points of the action which, in turn, is the
Euler characteristic of $M_\Delta$.

Therefore the inequality 
\begin{equation}\label{upper bound d refl}
d_\Delta \leq P(n,|V|)\cdot\left(\frac{n^4|V|}{2}\right)^{n|V|}
\end{equation}
comes from Corollary \ref{cor: box} and
 Remark \ref{rem better bound}, where equation \eqref{expl def chi} gives
 $$
 P(n,|V|)=\frac{n^2(n+3)|V|}{4}\,.
 $$
\end{proof}

\begin{rem}\label{comp Lagarias Ziegler}
For (not necessarily smooth) reflexive polytopes $\Delta$, a bound on their diameter is already known and is due to Lagarias
and Ziegler \cite{MR1138580}. Indeed, by combining Theorem 1 and 2 of their paper,  one can easily see that
 \[d_{\Delta}\leq  n\cdot n!\cdot 14^{n2^{n+1}}.\]
 Our bound \eqref{upper bound d refl} is different, as it is also a function of the number of vertices $|V|$ of $\Delta$.
\end{rem}

We recall that, for a smooth polytope $\Delta$, its $h$-vector $\mathbf{h}=(h_0,h_1,\ldots,h_n)$ corresponds to the vector $\mathbf{b}$ of even Betti numbers of the corresponding (symplectic toric) manifold (see for instance \cite[Lemma 3.8]{MR3695881}). We define
\begin{equation}\label{definition Cnh}
C(n,\mathbf{h}):=\sum_{j=0}^n h_{j}\Big[ 6j(j-1)+\frac{5n-3n^2}{2}\Big]\,.
\end{equation}
Given the discussion above, the following is an immediate consequence of Theorem \ref{thm: mainthm}.
\begin{cor}\label{cor reflexive}
Let $n\in \NN$, $n\geq 3$, and $\mathbf{h}\in \NN^{n+1}$. 
Then for each $n$ and $\mathbf{h}$, there are only finitely many possible values for the line Chern numbers of a smooth
reflexive polytope of dimension $n$ and with $h$-vector $\mathbf{h}=(h_0,\ldots,h_n)$. 

More precisely, let $C(n,\mathbf{h})$ be the constant defined in \eqref{definition Cnh}. 
Then the following bounds hold for any line Chern number $c_e(f)$ of pairs of distinct adjacent edges $e,f$:
\begin{equation}\label{bound cef reflexive}
(3-n)C(n,\mathbf{h})-2\leq c_e(f)< C(n,\mathbf{h})
\end{equation}
\end{cor}
\begin{rem}
The inequalities in \eqref{bound cef reflexive} can be stated equivalently in terms of the $f$-vector $\mathbf{f}=(F_0,\ldots,F_n)$, where
$F_i$ denotes the number of faces of dimension $i$. Indeed $\mathbf{h}$ can be expressed in terms of $\mathbf{f}$ and vice versa, and
using this correspondence it is easy to prove that $$C(n,\mathbf{h})=12F_2+(5-3n)F_1\,,$$ and therefore an equivalent formulation of \eqref{bound cef reflexive} is
\begin{equation}\label{bound cef reflexive f}
(3-n)[12F_2+(5-3n)F_1]-2\leq c_e(f)< 12F_2+(5-3n)F_1\,.
\end{equation}
\end{rem}

Corollary \ref{global bound} has the following immediate consequence

\begin{cor}\label{global bound reflexive}
Let $\Delta$ be a smooth reflexive polytope of dimension $n\geq 3$ with $|V|$ vertices. 
Then the following bounds hold for the line Chern numbers of any pair of distinct adjacent edges $e,f$ of $\Delta$:
\begin{equation}
\label{eq:global bound reflexive} (3-n)|V|\cdot \frac{n(n+1)^2}{2}-2\leq c_e(f)< |V|\cdot\frac{n(n+1)^2}{2}
\end{equation}

\end{cor}
\begin{rem}\label{reflexive dimension}
A celebrated result of Haase and Melnikov \cite{HaaseMelnikov2006} asserts that every lattice polytope $P$
 is lattice isomorphic to a face of some reflexive polytope $\Delta$ of sufficiently high dimension; the smallest such dimension
 is called the \emph{reflexive dimension} of $P$. 
 If one asks whether a lattice polytope $P$ can be the face of a \emph{smooth} reflexive polytope $\Delta$ of a certain dimension, then
 $P$ needs to be necessarily smooth (see for instance \cite[Lemma 2.8]{MR3695881}). Then Corollaries \ref{cor reflexive} and \ref{global bound reflexive} 
 would give restrictions on the line Chern numbers of $P$ (which are the same as those of $\Delta$ of the corresponding edges)
  in terms of the global properties of $\Delta$, namely in terms of its $\mathbf{h}$- or $\mathbf{b}$-vector, 
 or simply in terms of the number of vertices of $\Delta$. We wonder whether this observation could be used to 
 improve existing estimates on the reflexive dimension of a smooth polytope.  
 \end{rem}

\section{Open questions}\label{open questions}
We conclude the introduction with a list of open questions, that we believe are interesting for the communities of symplectic and algebraic geometers interested
in positive monotone Hamiltonian $T$-spaces and Fano varieties. 
\\$\;$\\
\textbf{Questions}:
\vspace{0.3cm}
\begin{itemize}
\item[(D)] Is there an \emph{upper bound} for the \emph{Euler characteristic} of a positive monotone Hamiltonian $T$-space with isolated fixed points in each dimension? 
\item[(E)] Are there \emph{finitely many (equivariant) homotopy/homeomorphism/diffeomorphism/\\symplectomorphism types} of compact positive monotone Hamiltonian $T$-spaces with isolated fixed points in each dimension?
\item[(F)] Is a compact positive monotone Hamiltonian $T$-space with isolated fixed points always (equivariantly) homotopy equivalent/homeomorphic/diffeomorphic/\\symplectomorphic to a Fano variety (endowed with a known action)?
\end{itemize}

In the above questions we restricted our attention to the case in which the fixed points of the action are isolated.  
This condition could be relaxed and replaced, for instance, by the assumption that $\pi_1(M)$ is trivial.
As for upper bounds on the Betti numbers, and therefore for the Euler characteristic, 
to the best of our knowledge the first bounds (depending on the index $k_0$ of
$(M,\omega)$) can be found in \cite[Section 5]{MR3695881} (see also \cite[Section
4]{MR4940256}).

Questions (E) and (F) were answered in the positive in \cite{MR5005925} for positive monotone Hamiltonian GKM spaces of dimension 6. In \cite{MR5060441} the authors classify 
positive monotone tall complexity one spaces of any dimension, answering in particular Questions (D), (E) and (F) in the affirmative. 
Question (F) is a generalization of the Fine--Panov conjecture (see 
\cite[Conjecture 1.4]{MR3355122}), which is stated for Hamiltonian $S^1$-spaces of dimension 6. It is not reasonable to believe that the answer to Question (F) is always affirmative,
but it would be interesting to know under which additional assumptions it holds true.

\vspace{0.5cm}

One fact underlying many classification results
is that the Chern number $c_1c_{n-1}[M]$ only depends on the Hirzebruch genus of $(M,\omega,\mu)$ which, in turn, only depends on the Betti numbers of $M$.
This is the key fact proving \cite[Corollary 3.1]{MR3230015}.  The inspiration of this result came from a result of Libgober and Wood \cite[Theorem 3]{MR1064869}, who prove that for a complex manifold $X$, $c_1c_{n-1}[X]$ only depends on the Hirzebruch genus which, in this case, only depends on the Hodge numbers of $X$.

For certain types of Hamiltonian actions (namely those admitting a \emph{toric one skeleton}, see \cite[Sections 3 and 4]{MR3695881}), there exists a special set of embedded symplectic spheres whose Poincar\'e dual is the Chern class $c_{n-1}$. 
If the Hamiltonian space admits a toric one skeleton and is positive monotone, the sum of the symplectic volumes of these spheres (also called \emph{magnitudes}) is therefore a topological invariant, more precisely it only depends on the even Betti numbers (see \cite[Theorem 1.5]{MR3695881}). 
This result is the first step for proving the finiteness of the line Chern numbers which, as explained before, proves the finiteness of the number of complex cobordism classes
in Theorem \ref{thm: finiteness intro}.
Moreover, it gives bounds and restrictions on the possible Betti numbers that can arise, which strongly depend on the index $k_0$ of the symplectic manifold (see \cite[Section 5.1]{MR3695881} and \cite[Section 4]{MR4940256}).
\vspace{0.5cm}

We wonder if this strategy could be applied to the \textbf{classification of Fano varieties}. \\$\;$\\
\textbf{Question}:
\vspace{0.3cm}
\begin{itemize}
\item[(G)] Are there special Fano varieties that admit a \emph{chain of rational curves}, such that the Poincar\'e dual to it is the Chern class $c_{n-1}$?
\end{itemize}
\vspace{0.3cm}
If the answer to this question is affirmative, the Chern number $c_1c_{n-1}[X]$ would be the sum of the volumes $\int_{S}\omega=c_1[S]$ of these rational curves $S$ and, therefore, it would be a positive integer
divisible by the index of the Fano variety. This, in turn, would give inequalities in the Hodge numbers depending on the index, in the spirit of those obtained in \cite[Section 5.1]{MR3695881}. 

Continuing the analogy with our strategy, assume that such a chain of rational curves exists and, on each of those, consider the line bundles in which the tangent bundle of $X$ splits.
The integrals of the first Chern classes of such line bundles on the rational curves give rise to integers, which in our case correspond to the line Chern numbers. We use the same name for 
the Fano case.
\\$\;$\\
\textbf{Question}:
\vspace{0.3cm}
\begin{itemize}
\item[(H)] Suppose that the Fano variety admits a chain of rational curves as in Question (G). Do the
corresponding line Chern numbers determine the complex cobordism class of the Fano variety? 
Are they bounded? 
\end{itemize}
$\;$\\

\bibliography{estimates}
\bibliographystyle{amsplain}

\end{document}